\documentclass[a4paper]{article}

\usepackage[latin1]{inputenc}
\usepackage[T1]{fontenc}
\usepackage{amssymb,amsmath,amsthm,amscd,mathrsfs}
\usepackage[all]{xy}
\usepackage{color}
\usepackage{float}
\usepackage{array}
\usepackage[margin=3cm]{geometry}

\newtheorem{thm}{Theorem}[section]
\newtheorem{defi}[thm]{Definition}
\newtheorem{prop}[thm]{Proposition}
\newtheorem{lemme}[thm]{Lemma}
\newtheorem{cor}[thm]{Corollary}

\theoremstyle{remark}

\newtheorem{rmk}[thm]{Remark}
\newtheorem{nota}[thm]{Notation}

\DeclareMathOperator{\Z}{\mathbb{Z}}
\DeclareMathOperator{\sing}{Sing}
\DeclareMathOperator{\codim}{Codim}
\DeclareMathOperator{\Hom}{Hom}

\DeclareMathOperator{\Pic}{Pic}

\DeclareMathOperator{\Fix}{Fix}

\DeclareMathOperator{\Vect}{Vect}
\DeclareMathOperator{\tors}{tors}

\DeclareMathOperator{\Ker}{Ker}

\DeclareMathOperator{\Sing}{Sing}

\DeclareMathOperator{\Q}{\mathbb{Q}}

\DeclareMathOperator{\C}{\mathbb{C}}
\DeclareMathOperator{\R}{\mathbb{R}}
\DeclareMathOperator{\dif}{d}

\DeclareMathOperator{\Ree}{Re}
\DeclareMathOperator{\Ima}{Im}

\DeclareMathOperator{\DR}{R}
\DeclareMathOperator{\Def}{Def}
\DeclareMathOperator{\NS}{NS}
\DeclareMathOperator{\Sp}{Sp}
\DeclareMathOperator{\inv}{inv}

\newcommand{\eq}[1][r]
{\ar@<-3pt>@{-}[#1]
\ar@<-1pt>@{}[#1]|<{}="gauche"
\ar@<+0pt>@{}[#1]|-{}="milieu"
\ar@<+1pt>@{}[#1]|>{}="droite"
\ar@/^2pt/@{-}"gauche";"milieu"
\ar@/_2pt/@{-}"milieu";"droite"}

\makeatletter
\newcommand{\opnorm}{\@ifstar\@opnorms\@opnorm}
\newcommand{\@opnorms}[1]{%
  \left|\mkern-1.5mu\left|\mkern-1.5mu\left|
   #1
  \right|\mkern-1.5mu\right|\mkern-1.5mu\right|
}
\newcommand{\@opnorm}[2][]{%
  \mathopen{#1|\mkern-1.5mu#1|\mkern-1.5mu#1|}
  #2
  \mathclose{#1|\mkern-1.5mu#1|\mkern-1.5mu#1|}
}
\makeatother
\begin{document}
\title{\bf Global Torelli theorem for irreducible symplectic orbifolds}

\author{Grégoire \textsc{Menet}}

%  \thanks{
%hyper-K\"ahler varieties, orbifolds, Betti numbers, singularities.}  

\maketitle

\let\thefootnote\relax
\let\thefootnote\relax

\begin{abstract}
We propose a generalization of Verbitsky's global Torelli theorem in the framework of compact Kähler irreducible holomorphically symplectic orbifolds by adapting Huybrechts' proof \cite{Huybrechts6}. As intermediate step needed, we also provide a generalization of the twistor space and the projectivity criterion based on works of Campana \cite{Campana} and Huybrechts \cite{Huybrechts2} respectively. 
\end{abstract}

\section{Introduction}
In the last several years, \emph{irreducible holomorphically symplectic} (IHS) manifolds became a very important topic of study. However, the first fundamental contribution to this topic, after Bogomolov's decomposition theorem, was established in a more general setting by Fujiki in \cite{Fujiki}. Fujiki considered irreducible holomorphically symplectic orbifolds. A complex analytic space is called an \emph{orbifold} if it is locally isomorphic to a quotient of an open subset of $\C^n$ by a finite automorphism group (Definition \ref{VV}). Let $X$ be an orbifold, $X$ is said \emph{irreducible holomorphically symplectic} if $W:=X\smallsetminus \Sing X$ is simply connected and admits a unique, up to a scalar multiple, nondegenerate holomorphic 2-form (Definition \ref{VS}). In this paper, we call these objects, for simplicity, irreducible symplectic orbifolds or (IHS) orbifolds.

Since this first contribution of Fujiki, the development of this field was concentrated on smooth (IHS) manifolds with fundamental results by Beauville \cite{Beauville} and Huybrechts \cite{Huybrechts5}. Nevertheless, coming back to the original idea of Fujiki is a very promising direction of research. On the one hand, the numerical limitation of examples disappears with (IHS) orbifolds. 
%there is no limit in terme of quantity of example that we can found are many more examples of (IHS) orbifolds than for its smooth cousin. 
Markushevich, Tikhomirov and Matteini in \cite{Markou} and \cite{Matteini}, provided examples coming from moduli spaces of stable sheaves on some particular K3 surfaces; in addition, several examples coming from quotients were studied in \cite[Section 13]{Fujiki}, \cite{Lol}, \cite{Kapfer} and \cite{Lol2}. On the other hand, it is quite simple to generalize the theory of smooth (IHS) manifolds to the case of (IHS) orbifolds. Many tools were already generalized in the literature and (IHS) orbifolds seem to behave in many aspects as their smooth cousins. It is also an objective of this paper to back up this assertion. The most remarkable properties are the existence of a pure Hodge structure (Theorem \ref{Hodge}), the stability under deformation (Proposition \ref{triv}) and the generalization by Campana of Calabi--Yau's theorem (Theorem \ref{Ricci}). After the global Torelli theorem, many more generalizations could be expected, and this paper can be seen as a first contribution to motivate future research in this direction. The global Torelli theorem will also open many perspectives of research in the framework of (IHS) orbifolds, for instance to the study of mirror symmetry. This work also takes place in the more general project of generalization of theories coming from Kähler manifolds to Kähler complex spaces which will be more adapted to the minimal model program. As a similar initiative I would like to mention the work of Benjamin Bakker and Chistian Lehn which provided a global Torelli theorem for singular symplectic varieties admitting a symplectic resolution \cite{Lehn}. 

The second cohomology group $H^2(X,\Z)$ of an (IHS) orbifold $X$ can be endowed with its \emph{Beauville--Bogomolov form} $B_X$ (cf. Section \ref{TorelliSection} for the definition) of signature $(3,b_2(X)-3)$. Let $\Lambda$ be a lattice of signature $(3,b-3)$, with $b\geq3$, we call a \emph{marking} of an (IHS) orbifold $X$ an isometry $\varphi: H^2(X,\Z)\rightarrow \Lambda$ where $H^2(X,\Z)$ is endowed with the Beauville--Bogomolov form. We can construct the moduli space $\mathscr{M}_{\Lambda}$ of marked (IHS) orbifolds of the Beauville--Bogomolov lattice $\Lambda$ (see Section \ref{modu}). As in the smooth case, the global \emph{period map} is defined by:
\begin{align*}
\mathscr{P}:\ & \mathscr{M}_{\Lambda}\rightarrow \mathcal{D}=\mathbb{P}\left(\left\{\sigma\in \Lambda\otimes\mathbb{C} |\ \sigma^{2}=0,\ (\sigma+\overline{\sigma})^{2}>0\right\}\right)\\
& (X,\varphi)\mapsto \varphi(H^{2,0}(X)).
\end{align*}
Moreover, the moduli space $\mathscr{M}_{\Lambda}$ is not Hausdorff, however, we can construct its Hausdorff reduction $\overline{\mathscr{M}_{\Lambda}}$ (see Corollary \ref{Hausdorff}) through which the period map factorizes: 
$$\mathscr{P}: \mathscr{M}_{\Lambda}\twoheadrightarrow \overline{\mathscr{M}_{\Lambda}}\rightarrow \mathcal{D}.$$
Then the global Torelli theorem (Corollary \ref{torelliglobal}) can be expressed as follows:
\begin{thm}
Let $\Lambda$ be a lattice of signature $(3,b-3)$, with $b\geq3$. Assume that $\mathscr{M}_{\Lambda}\neq\emptyset$ and let $\mathscr{M}_{\Lambda}^{°}$ be a connected component of $\mathscr{M}_{\Lambda}$. Then the period map:
$$\mathscr{P}: \overline{\mathscr{M}_{\Lambda}}^{°}\rightarrow \mathcal{D}$$
is an isomorphism. 
\end{thm}
To prove this theorem, several basic tools were also generalized such as the projectivity criterion:
%\begin{itemize}
%\item
%\emph{Let $X$ be a primitively symplectic orbifold. There exists a divisor $D$ on $X$ such that $q_X(D)>0$ if and only if $X$ is projective.}
%\end{itemize}
\begin{thm}\label{PC}
Let $X$ be a primitively symplectic orbifold. There exists a line bundle $\mathcal{L}$ on $X$ such that $B_X(c_1(\mathcal{L}),c_1(\mathcal{L}))>0$ if and only if $X$ is projective.
\end{thm}
The notion of \emph{primitively symplectic orbifold} (introduced for the first time by Fujiki in \cite{Fujiki}) is a weaker version of the notion of irreducible symplectic orbifold where the simple connectedness of the smooth locus is omitted (Defintion \ref{VS}). The projectivity criterion is in particular needed to understand the Käher cone of a general (IHS) orbifolds (cf. \cite[Proposition 5.1]{Huybrechts5} and Corollary \ref{Kählercone}).
%\begin{itemize}
%\item
%\emph{Assume that $\Pic X=0$, then 
%$\mathcal{K}_{X}=\mathcal{C}_X.$}
%\end{itemize}
Another important ingredient is the generalization of the \emph{twistor space} (cf. Section \ref{twisty}).
Using the fact that orbifolds are locally quotients of smooth open sets by finite automorphism groups, we can generalize the definition of several objects of Riemannian geometry such as metrics or complex structures, etc. 
Using results of Campana \cite{Campana}, we can show that for an (IHS)-orbifold $X$ and a Kähler class $\alpha$ on $X$,
we can find a Ricci flat metric $g$ and three complex structures in quaternionic relations $I$, $J$ and $K$ on $X$ with $\alpha=[g(\cdot,I\cdot)]$.
Then as in the smooth case, the twistor space is the deformation parameterizing the complex structures $aI+bJ+cK$ with $(a,b,c)\in S^2$.

%This paper can be seen as a first contribution  before a generalization of the global Torelli theorem; we provide a necessary basic tool: the \emph{projectivity criterion}.
%\begin{thm}\label{Main}
%Let $X$ be a compact Kähler irreducible symplectic orbifold with $\codim\Sing X\geq4$. There exists a divisor $D$ on $X$ such that $q_X(D)>0$ if and only if $X$ is projective.
%\end{thm}

%An orbifold is said \emph{Kähler} if it is Kähler as a complex analytic space (see \cite[Definition 1.2]{Fujiki2}). The object $q_X$ is a quadratic form on $H^2(X,\C)$ integral on $H^2(X,\Z)$ which is called the \emph{Beauville--Bogomolov form} (cf. Section \ref{TorelliSection} for the definition). The Beauville--Bogomolov form is a fundamental topological invariant inescapable for the study of (IHS) manifolds and omnipresent in this paper.
%Another important steps to the global Torelli theorem is to study the non-separated points of the moduli spaces of marked (IHS) orbifolds $\mathscr{M}_{\Lambda}$. We prove, as in the smooth case, that non-separated points in $\mathscr{M}_{\Lambda}$ provide bimeromorphic orbifolds and we conclude by providing a Hausdorff reduction of $\mathscr{M}_{\Lambda}$ (cf. Proposition \ref{separated} and Corollary \ref{Hausdorff}).

The paper is organized as follows. In Section \ref{RandC}, we provide some reminders and complements such as de Rham's theorems and Hodge decomposition for orbifolds. In particular, we generalize the Hodge--Riemann relation and the Lefschetz (1,1) theorem. Section \ref{Moduli} is dedicated to the construction of the moduli space $\mathscr{M}_{\Lambda}$ of Kähler (IHS) orbifolds. We propose a complete proof which is as simple as possible and as detailed as necessary of the Fujiki relation and local Torelli theorem (Theorem \ref{LocalTorelli}). The objective is to provide a survey of several ideas in the topic to start the theory on a solid and clear grounds. Also, adapting Huybrechts' proof \cite[Theorem 4.3]{Huybrechts5}, we describe the non-separated points of the moduli space in Proposition \ref{separated}. This proposition is applied to provide a Hausdorff reduction of $\mathscr{M}_{\Lambda}$ in Corollary \ref{Hausdorff}. Section \ref{MainSection} is dedicated to the proof of the projectivity criterion (Theorem \ref{PC}) which is an adaptation of another proof of Huybrechts (\cite{Huybrechts2}). Finally, Section \ref{Gloglo} adapts Huybrechts' proof of the global Torelli theorem (\cite{Huybrechts6}). In particular, we propose a genearalization of the twistor space in Section \ref{twisty}.

~\\

\textbf{Acknowledgements.} 
I am very grateful to Ulrike Rie\ss\ for very profitable discussions and many important comments. I want to thank Daniel Huybrechts for very useful discussions, Daniele Faenzi, Dimitri Markushevich, and Lucy Moser-Jauslin for helpful comments, Arvid Perego for his remark about examples of irreducible symplectic orbifolds, Frederic Campana for a very encouraging emails exchange and Dan Zaffran for his reference about Kähler orbifolds. I am also very grateful to my referee who provides a high quality report pointing out an important confusion in the previous version of this paper. This work was supported by the Marco Brunella grant of Burgundy University and the ERC-ALKAGE, grant No. 670846.
%\section{Deformation of Kähler irreducible symplectic orbifolds}
\section{Reminders and complements}\label{RandC}
\subsection{Definition of orbifolds}
%For more clearness, we recall the full definition of orbifolds, notion introduced for the first time in \cite{Satake}. 
%\begin{defi}
%A $\mathcal{C}^\infty$ (resp. complex) orbifold is a topological space endowed with a basis of open $\mathscr{B}$ such that:
%\begin{itemize}
%\item[(i)]
%For all $U\in \mathscr{B}$ there exists $(\widetilde{U},G,f)$ where $\widetilde{U}$ is an open of $\mathbb{R}^n$ (resp. $\mathbb{C}^n$), $G$ a finite group of automorphisms of $\widetilde{U}$ and $f:\widetilde{U}/G\rightarrow U$ a homeomorphism. The data $(\widetilde{U},G,f)$ is called a \emph{local uniformizing system} (l.u.s.) of $U$.
%\item[(ii)]
%Let $U$ and $U'$ be two open sets in $\mathscr{B}$ such that $U'\subset U$. Let $(\widetilde{U},G,f)$ and $(\widetilde{U}',G',f')$ be local uniformizing system for $U$ and $U'$ respectively. There exists a $\mathcal{C}^\infty$ (resp. holomorphic) embedding $\lambda:\widetilde{U}'\rightarrow \widetilde{U}$ such that for all $g\in G$ there exists $g'\in G'$ such that 
%\end{itemize}
%\end{defi}
\begin{defi}\label{VV}
A $n$-dimentional \emph{orbifold} is a connected paracompact complex space $X$ such that for every point $x\in X$, there exists an open neighborhood $U$ and a triple $(V,G,\pi)$ such that 
$V$ is an open set of $\C^n$, $G$ is a finite automorphism group of $V$ and $\pi:V\rightarrow U$ the composition of the quotient map $V\rightarrow V/G$ with an isomorphism $V/G\simeq U$. Such a triple $(V,G,\pi)$ is called a \emph{local uniformizing system} of $x$ or of $U$. 
\end{defi}
\begin{nota}\label{XUV*}
Let $X$ be an orbifold and $(V,G,\pi)$ a local uniformizing system of an open set $U$. We denote $X^*=X\smallsetminus \Sing X$, $U^*=U\smallsetminus \Sing U$ and $V^*=V\smallsetminus \pi^{-1}(\Sing U)$. We will also refer about $\pi$ as the quotient map. 
%We will also refer to $f\circ \pi:V\rightarrow U$ as the quotient map and simply denote it by $\pi$. 
\end{nota}
\begin{rmk}\label{PrillR}
Let $x\in X$. Remark from \cite[Proposition 6]{Prill}, that we can always find a neighborhood $U$ of $x$ with a local uniformizing system
$(V,G,f)$ such that for all $g\in G$, $\codim \Fix g\geq 2$ and so $\pi:V^*\rightarrow U^*$ is an étale cover.
From now, local uniformizing systems will always have these properties. 
%Let $X$ be a complex orbifold and $x\in X$. Remark that we can always choose a local uniformizing system $(V,G,f)$ of $x$ such that $\codim 
\end{rmk}
%In all this paper, we consider complex \emph{orbifolds}, that is complex analytic spaces locally isomorphic to a quotient of an open subset of $\C^n$ by a finite automorphisms group. 
We recall that the quotient singularities are mild (see for instance \cite[Proposition 1.3]{Blache} and \cite[Proposition 5.15]{Kollar}).
\begin{prop}\label{MildSingu}
Let $X$ be an orbifold. Then $X$ is normal, $\Q$-factorial, Cohen-Macaulay, with only rational singularities.
\end{prop}
\subsection{Differential forms on orbifolds}\label{dif}
Satake in \cite{Satake} defines the differential forms on an orbifold $X$ of dimension $n$ as follows. 
Let $(U_i)$ be an open cover of $X$ such that each $U_i$ admits a local uniformizing system $(V_i,G_i,\pi_i)$.
%$U_i=V_i/G_i$ where $V_i$ is an open subset of $\C^n$ and $G_i$ an automorphism group on $V_i$.
Let $\mathbb{K}$ be the field $\R$ or $\C$.
We define the \emph{sheaf of differential forms} of degree $d$ on $U_i$ by:
$$\mathcal{A}^d_{X|U_i}(\mathbb{K}):=\pi_{i*}(\mathcal{A}^d_{V_i}(\mathbb{K})^{G_i}),$$
where $\mathcal{A}^d_{V_i}(\mathbb{K})^{G_i}$ are the $\mathscr{C}^{\infty}$-forms of degree $d$ on $V_i$ with values in $\mathbb{K}$ invariant under the action of $G_i$. Satake in \cite{Satake} shows that the $\mathcal{A}^d_{X|U_i}(\mathbb{K})$ can be glue in a sheaf on $X$ that we denote by $\mathcal{A}^d_{X}(\mathbb{K})$. For simplicity in the notation, when $\mathbb{K}=\C$
% and there is no ambiguity, 
we will omit the coefficient field. The same construction is possible for the \emph{holomorphic $p$-forms} and we denote the associated sheaf by $\Omega_X^p$. 

Satake in \cite{Satake} generalizes the de Rham theorems (see \cite[Theorem 1.9]{Blache} for the proof of (iii)).
\begin{thm}\label{de Rham}
Let $X$ be a compact orbifold of dimension $n$ and $k\in\left\{0,...,2n\right\}$.
\begin{itemize}
\item[(i)]
The map $H^k_{dR}(X)\rightarrow \Hom(H_k(X),\R)$ given by $\varphi\mapsto(c\mapsto\int_c \varphi)$ is an isomorphism. 
\item[(ii)]
The map $H^k_{dR}(X)\rightarrow H^{2n-k}_{dR}(X)^{\vee}$ given by $\varphi\mapsto(\psi\mapsto\int_X \psi\cdot\varphi)$ is an isomorphism.
\item[(iii)]
The map $H^*_{dR}(X)\simeq \Hom(H_*(X),\R)\simeq H^*(X,\R)$ induced from (i) is a ring isomorphism where $H^*(X,\R)$ is endowed with the cup-product and $H^*_{dR}(X)$ with the wedge product. 
\end{itemize}
\end{thm}
\begin{rmk}\label{Inject}
Let $r:\widetilde{X}\rightarrow X$ be a resolution of an orbifold $X$. Blache in \cite[Section 1.15]{Blache} proves that the map $r^*:H^*_{dR}(X)\rightarrow H^*_{dR}(\widetilde{X})$ can be defined directly using currents on $\widetilde{X}$ and moreover that it is an injective map. 
\end{rmk}
We can deduce from Theorem \ref{de Rham} the Poincaré duality with rational coefficients.
\begin{cor}\label{poincaréQ}
Let $X$ be a compact orbifold of dimension $n$ and $k\in\left\{0,...,2n\right\}$.
The cup product provides a perfect pairing:
$$H^k(X,\Q)\times H^{2n-k}(X,\Q)\rightarrow H^{2n}(X,\Q)\simeq \Q.$$
\end{cor}
\begin{proof}
First of all, we remark that (ii) and (iii) of Theorem \ref{de Rham} provide a perfect pairing given by the cup product with real coefficients:
\begin{equation}
H^k(X,\R)\times H^{2n-k}(X,\R)\rightarrow H^{2n}(X,\R)\simeq \R.
\label{poincaréReal}
\end{equation}
Moreover from the universal coefficient theorem, we have:
\begin{equation}
H^l(X,\R)=H^l(X,\Q)\otimes\R,
\label{coeffRQ}
\end{equation}
for all $l\in \left\{0,...,2n\right\}$.
Now consider the natural map provided by the cup product:
$$D_X:H^{k}(X,\Q)\rightarrow H^{2n-k}(X,\Q)^{\vee}.$$
From (\ref{poincaréReal}) and (\ref{coeffRQ}), $D_X$ is injective.
Then, for a reason of dimension, $D_X$ is also surjective. 
\end{proof}
\begin{rmk}\label{projformula}
%Combining (i) and (ii), we obtain the \emph{Poincaré duality} with real coefficients: $D_X:H^k(X,\R)\simeq H^{2n-k}(X,\R)^\vee$. 
It allows us to generalize the definition of the \emph{push-forward map}. Let $f:Y\rightarrow X$ be a continuous map between orbifolds. We can define $f_*:H^*(Y,\Q)\rightarrow H^*(X,\Q)$ by $f_*(\alpha):=D^{-1}_X(D_Y(\alpha)\circ f^*)$. Moreover, the \emph{projection formula} generalizes as well. 
\end{rmk}
\subsection{Hodge decomposition of Kähler orbifolds}
\begin{defi}
Let $X$ be an orbifold. A form $\omega\in \mathcal{A}^2_X(X,\R)$ is said \emph{Kähler} if for all local uniformizing system $(V,G,\pi)$ of an open set $U$, the form
$\omega_{X|U}\in \mathcal{A}^2_X(U,\R)=(\mathcal{A}^2_V)^G(V,\R)$ is a Kähler form on $V$. An orbifold which admits a Kähler form is called a \emph{Kähler orbifold}. 
\end{defi}
\begin{rmk}
An orbifold is Kähler if and only if it is Kähler as a complex space (cf. \cite[Definition 1.2, 4.1 and Remark 4.2]{Fujiki2} for the definition of a Kähler complex space and \cite[pages 793-795]{Zaffran} for the proof of the equivalence). There is a third equivalent definition for a Kähler form on $X$. A Kähler form on $X$ is a Kähler form $\omega$ on $X^*$ such that for all local uniformizing system $(V,G,\pi)$ of an open set $U$, the form $\pi^*\omega_{|U^*}$ extends to a Kähler form on $V$.
\end{rmk}
From now on $X$ is a compact Kähler orbifold of dimension $n$.
First, we recall from \cite[Section 2.5]{Peters} (see also \cite[Theorem 1.12]{Steenbrink}) that a Kähler orbifold admits a \emph{pure Hodge structure}. 
Let $j:X^*\hookrightarrow X$ be the embedding. Let $\Omega_X^p$ be the sheaf of holomorphic p-forms as defined in Section \ref{dif}. 
Peters and Steenbrink first show that: 
\begin{equation}
\Omega_X^p\simeq j_*(\Omega_{X^*}^p),
\label{PetersLemma}
\end{equation}
and deduce that $\Omega_X^\bullet$ is a resolution of the constant sheaf $\underline{\C}_X$.
%\begin{lemme}[\cite{Peters}, Lemma 2.46]\label{}
%There is an isomorphism:
%$$\Omega_X^p\simeq j_*(\Omega_{X^*}^p).$$
%\end{lemme}
Then, they consider the spectral sequence of hypercohomology:
$$E_1^{p,q}=H^q(X,\Omega_X^p)\Rightarrow \mathbb{H}^{p+q}(X,\Omega_X^\bullet)=H^{p+q}(X,\C).$$
Let $F^p H^k(X,\C)$ be the hodge filtration, we denote $H^{p,q}(X):=F^p H^{p+q}(X,\C)\cap \overline{F^q H^{p+q}(X,\C)}$.
Peters and Steenbrink show that the spectral sequence degenerates at the $E_1$ page and that $H^{p,q}(X)\simeq H^q(X,\Omega_X^p)$.
From this theorem and its proof we can deduce the following result which will be used several times in this paper:
%Peters and Steenbrink in \cite[Section 2.5]{Peters} deduce the following theorem. 
\begin{thm}\label{Hodge}
Let $X$ be a compact Kähler orbifold of dimension $n$.
\begin{itemize}
\item[(i)]
We have the Hodge decomposition:
$$H^k(X,\C)=\bigoplus_{p+q=k} H^q(X,\Omega_X^p).$$
\item[(ii)]
The natural pairing 
$H^q(X,\Omega_X^p)\times H^{n-q}(X,\Omega_X^{n-p})\rightarrow H^{n}(X,\Omega_X^{n})\simeq \C$
is perfect.
\item[(iii)]
Moreover, if $r:\widetilde{X}\rightarrow X$ is a Kähler resolution of singularities of $X$, then $r^*:H^*(X,\C)\rightarrow H^*(\widetilde{X},\C)$ is an injective morphism of Hodge structures.
\end{itemize}
\end{thm}
%We remark from (iii) that, we can work with holomorphic forms on $X$ thinking as forms of same type on a resolution $\widetilde{X}$. In particular, we can extend the Hodge--Riemann relation for Kähler orbifolds. 
%We denote as usual $H^{p,q}(X):=H^q(X,\Omega_X^p)$.
Let $Z$ be an analytic subset of $X$ of codimension $p$. We have seen in Section \ref{dif} that from de Rham's theorems (Theorem \ref{de Rham} (i) and (ii)), the class $[Z]\in H_{2n-2p}(X)$ provides a differential de Rham class $[Z]\in H_{dR}^{2p}(X)$.  
\begin{prop}\label{typepp}
Let $Z$ be an irreducible analytic subset of $X$. 
Let $[Z]$ be the class associated to $Z$ in $H_{dR}^{2p}(X)$, then $[Z]\in H^{p,p}(X)\cap H^{2p}(X,\Q)$.
\end{prop}
\begin{proof}
We have $[Z]\in H^{2p}(X,\Q)$ by Corollary \ref{poincaréQ}.
\begin{itemize}
\item \textbf{First case}: $Z\not\subset \Sing X$. 

Let $r:\widetilde{X}\rightarrow X$ be a resolution of $X$. We can consider $Z'$ the strict transform of $Z$ in $\widetilde{X}$.
Let $\varphi\in H_{dR}^{2n-2p}(X,\C)$ such that $\varphi$ is of type $(a,b)\neq (n-p,n-p)$. 
We will prove that $\int_X[Z]\cdot \varphi=0$ and conclude using (ii) of Theorem \ref{Hodge}.
Using Theorem \ref{de Rham} and \cite[Theorem 11.21]{Voisin} we have:
\begin{align*}
\int_X[Z]\cdot \varphi&= \int_Z\varphi\\
&= \int_{Z'}r^*(\varphi)\\
&=\int_{Z'\smallsetminus\Sing Z'}r^*(\varphi).
\end{align*} 
From (iii) of Theorem \ref{Hodge}, $r^*(\varphi)$ is also of type $(a,b)\neq (n-p,n-p)$. 
It follows from an argument of dimension that:
$\int_{Z'\smallsetminus\Sing Z'}r^*(\varphi)=0$.
\item\textbf{Second case}: $Z\subset \Sing X$.

From \cite[Lemma 1.4]{Fujiki}, there is an orbifold $Y\subset \Sing X$ such that $Z\subset Y$ but $Z\not\subset \Sing Y$.
Let $j:Y\hookrightarrow X$ be the inclusion. By definition of the holomorphic forms in Section \ref{dif}, the morphism
$j^*:H^{2n-2p}(X,\C)\rightarrow H^{2n-2p}(Y,\C)$ is a morphism of Hodge structures. Let $\varphi\in H_{dR}^{2n-2p}(X,\C)$ such that $\varphi$ is of type $(a,b)\neq (n-p,n-p)$.
Using Theorem \ref{de Rham}, we have:
\begin{align*}
\int_X[Z]\cdot \varphi&= \int_Z\varphi\\
&=\int_{Z}j^*(\varphi).
\end{align*} 
\end{itemize}
Since $j^*$ is a morphism of Hodge structure, we are back to the first case. 
\end{proof}
%As a corollary, we can generalize the (1,1)-Lefschetz theorem.
%\begin{cor}
%Let $X$ be a compact Kähler orbifold. We denote by $N.S.(X)$ the sub-group of $H^{2}(X,\Z)$
%\end{cor}
In the case of the Néron-Severi group, we can be more precise and generalize the Lefschetz (1,1) theorem as a consequence of Theorem \ref{Hodge} (i). It will be necessary to deduce the Picard lattice of an irreducible symplectic orbifold from its period (cf. Section \ref{TorelliSection} and \ref{modu} for the definition of the period map). 
\begin{prop}\label{Lefschetz11}
Let $\NS(X)$ be the Néron-Severi group of $X$. Then:
$$\NS(X)\simeq H^{1,1}(X)\cap H^{2}(X,\Z).$$
\end{prop}
%Now, we are ready to prove the Lefschetz (1,1) theorem.
\begin{proof}
%[Proof of Proposition \ref{Lefschetz11}]
The proposition can be proved as in the smooth case.
The exact sequence of sheaves:
$$\xymatrix@C15pt{
 0\ar[r]&\Z\ar[r] &\ar[r]\mathcal{O}_X&\ar[r]\mathcal{O}_X^*&0,
    }$$
 induces an exact sequence of cohomology groups:
$$\xymatrix{
 H^1(X,\mathcal{O}_X^*)\ar[r]^{h} &\ar[r]H^2(X,\Z)&H^2(X,\mathcal{O}_X).
    }$$
The image of $h$ is the Néron-Severi group. However, the inclusion $\C\hookrightarrow \mathcal{O}_X$ induces the following exact sequence:
$$\xymatrix{
 0\ar[r]&F^1H^2(X,\C)\ar[r]&\ar[r]H^2(X,\C)&H^2(X,\mathcal{O}_X)\ar[r]&0,
    }$$
because $(\mathcal{A}^._{X},d)$ is a resolution of $\C$ and $(\mathcal{A}^{0,.}_{X},\overline{\partial})$ a resolution of $\mathcal{O}_X$.
Then, we finish the proof by noticing that the Hodge decomposition provides $F^1H^2(X,\C)\cap H^2(X,\Z)=H^{1,1}(X)\cap H^2(X,\Z)$.
\end{proof}
We finish this section by generalizing the Hodge--Riemann relation used in Section \ref{TorelliSection} for calculating the signature of the Beauville--Bogomolov form.
\begin{prop}\label{Hodge--Riemann}
Let $\omega$ be a Kähler form on $X$. We define the Lefschetz operator $L:H^{p,q}(X)\rightarrow H^{p+1,q+1}(X), \varphi\mapsto \omega\cdot \varphi$.
Assume $p+q\leq n$, we set $H^{p,q}(X)_p:=\Ker L^{n-p-q+1}$. Let $\alpha\in H^{p,q}(X)_p$, then: 
$$i^{p-q}(-1)^{\frac{(p+q)(p+q-1)}{2}}\int_X\alpha\cdot \overline{\alpha}\cdot\omega^{n-p-q}\geq0.$$
Moreover, we have an equality if and only if $\alpha=0$.
\end{prop}
To prove this proposition, we will need to consider a resolution of singularities. 
%To obtain a smooth Kähler manifold $\widetilde{X}$,
%we consider $r:\widetilde{X}\rightarrow X$ a resolution of $X$ obtained by a finite sequence of blow-ups, it exists by \cite[Theorem 13.2]{BM}.
\begin{lemme}\label{Kähler}
Let $\omega$ be a Kähler class on $X$. Let $r:\widetilde{X}\rightarrow X$ be a blow-up in a smooth subset with exceptional divisor $E$.  
Then there exists $R_0\in\mathbb{R}^*_+$ such that for all $R\geq R_0$, $Rr^*(\omega)-\left[E\right]$ is a Kähler class on $\widetilde{X}$.
\end{lemme}
\begin{proof}
This lemma is given by the proof of \cite[Proposition 4.6 (2) and (4)]{Bingener}.
% provides the following lemma (we can also adapt the proof given for the smooth case in \cite[Section 3.3.3]{Voisin}).

We can also adapt the proof of \cite[Section 3.3.3]{Voisin} given in the smooth case.
Let $Y\subset X$ be the center of the blow-up $r$.
We can find a open cover ($V_i$) of $X$ such that for each $i$, $V_i$ can be embedded as a closed space in a smooth manifold $M_i$.
Then the morphism $r: r^{-1}(V_i)\rightarrow V_i$ can be identified with the morphism $\widetilde{V_i}\rightarrow V_i$ induced by the blow-up $r_i:\widetilde{M_i}\rightarrow M_i$ in $Y\cap V_i$; $\widetilde{V_i}$ being the strict transform of $V_i$ by $r_i$. We also denote by $E_i$ the exceptional divisor of $r_i$. As explain in the proof of \cite[Proposition 3.24]{Voisin}, we can find a metric $h_i$ on $\mathcal{O}_{\widetilde{M_i}}(-E_i)$ such that the induced Chern form $\omega_i$ is strictly positive on the fibers of $E_i\rightarrow Y\cap V_i$ and zero outside a compact neighborhood of $E_i$. These metric $h_i$ restrict to metrics on $\mathcal{O}_{\widetilde{V_i}}(-E_i\cap \widetilde{V_i})$ 
which can be glued together into a metric $h$ on 
$\mathcal{O}_{\widetilde{X}}(-E)$. By construction of $h$, the associated Chern form $\omega_E$ will be strictly positive on each fibers of $E\rightarrow Y$ and zero outside a compact neighborhood of $E$. Then a compactness argument allows to conclude the proof.
%The proof in the smooth case () can be adapted in the singular case (see \cite[Lemma 4.4]{Fujiki2} and \cite[Lemma 2]{Fujiki3}).
\end{proof}
\begin{proof}[Proof of Proposition \ref{Hodge--Riemann}]
Let $\omega$ be a Kähler class on $X$. Let $r:\widetilde{X}\rightarrow X$ be a resolution of $X$ obtained by a finite sequence of blow-ups (see \cite[Theorem 13.2]{BM}).
Applying Lemma \ref{Kähler} recursively, we can find a real $R_0\in\mathbb{R}_+^*$ and a linear combination $E$ of exceptional divisors such that $\widetilde{\omega}:=r^*(R_0\omega)+\left[E\right]$ is a Kähler class on $\widetilde{X}$. We denote by $H^{p,q}(X)_p$ (resp. $H^{p,q}(\widetilde{X})_p$) the set of primitive classes associated to $\omega$ (resp. $\widetilde{\omega}$). Let $\alpha\in H^{p,q}(X)_p$, we can see that $r^*(\alpha)\in H^{p,q}(\widetilde{X})_p$.
Indeed, from Theorem \ref{Hodge} (iii), we know that $r^*$ is a morphism of Hodge structure. Moreover from the projection formula (see Remark \ref{projformula}):
$$\int_{\widetilde{X}}\left(r^*(R_0\omega)+\left[E\right]\right)^{n-p-q+1}\cdot r^*(\alpha)=\int_{\widetilde{X}} r^*(R_0\omega)^{n-p-q+1}\cdot r^*(\alpha)=\int_{X} (R_0\omega)^{n-p-q+1}\cdot \alpha=0.$$
Let denote $C:=i^{p-q}(-1)^{\frac{(p+q)(p+q-1)}{2}}$, the projection formula also provides:
$$C\int_{\widetilde{X}} r^*(\alpha)\cdot r^*(\overline{\alpha})\cdot r^*(R_0\omega)^{n-p-q}=C\int_{\widetilde{X}} r^*(\alpha)\cdot r^*(\overline{\alpha})\cdot\left(r^*(R_0\omega)+\left[E\right]\right)^{n-p-q}\geq0,$$
where the positivity comes from the Hodge--Riemann relation in the smooth case (see for instance \cite[Proposition 3.3.15]{Huybrechts3}).
We conclude by using the injectivity of the map $r^*$ (see Remark \ref{Inject}).
\end{proof}
\begin{rmk}
As explained in \cite{Steenbrink}, with the same method, we can prove the Hard Lefschetz theorem. 
\end{rmk}
\section{Moduli space of marked holomorphically symplectic orbifolds}\label{Moduli}
\subsection{Definition of symplectic orbifolds}
\begin{defi}\label{VS}
A compact Kähler orbifold $X$ with $\codim\Sing X\geq4$ is said \emph{symplectic} if there exists a non-degenerate holomorphic 2-form on the smooth locus $X^*$. The orbifold $X$ is said \emph{primitively symplectic} if, further, $h^{2,0}(X)=1$. A primitively symplectic orbifold $X$ is said \emph{irreducible symplectic} if $X^*$ is simply connected.   
\end{defi}
\begin{rmk}
The condition $\codim\Sing X\geq4$ is not restrictive. Indeed from \cite[Proposition 2.7]{Fujiki}, we cannot find any component of $\Sing X$ of codimension 3 and because of \cite[Corollary 1]{NotaNami} the components of $\Sing X$ of codimension 2 can be resolved such that the resolution is still a symplectic orbifold (in practice, we can use \cite[Proposition 2.9]{Fujiki}). 
\end{rmk}
\begin{rmk}
Actually, for many further results, simple connectedness is not needed. 
This is one of the reasons why we will use the notion of primitively symplectic orbifold considered by Fujiki (cf. \cite[Defintion 2.1]{Fujiki}).
%That is why, we will use the term of primitively symplectic orbifold considered by Fujiki (cf. \cite[Defintion 2.1]{Fujiki}).
\end{rmk}
\begin{rmk}
If $X$ is a symplectic orbifold and $r:\widetilde{X}\rightarrow X$ a resolution of singularities then a holomorphic 2-form on $X^*$ extends uniquely to a holomorphic 2-form on $\widetilde{X}$ (see \cite[Corollary 1.8]{Kebekus}).  
\end{rmk}
Irreducible symplectic orbifolds are of particular interest since they appear as elementary bricks in the generalization of the Bogomolov decomposition theorem by Campana.
\begin{thm}[\cite{Campana}, Théorème 6.4]\label{campapa}
Let $X$ be a compact Kähler orbifold with $c_1(X)=0$. Then there exists $\gamma:\bar{X}\rightarrow X$ a finite cover of orbifolds (see \cite[Definition 5.1]{Campana}) with:
$$\bar{X}=T\times X_1\times\cdots X_k\times Y_1\times \cdots Y_q,$$
where $T$ is a torus, the $X_i$ are irreducible symplectic orbifolds and the $Y_i$ are Calabi--Yau orbifolds (that is compact Kähler orbifolds with simply connected smooth locus, trivial canonical bundle and such that $h^{2,0}(Y_i)=0$).
\end{thm}
\begin{rmk}
%In particular, considering primitively symplectic orbifolds, we can 
In particular, we can apply the previous theorem to primitively symplectic orbifolds. Let $X$ be a primitively symplectic orbifold, by Theorem \ref{campapa}, there exists a finite cover of orbifolds $\gamma:\bar{X}\rightarrow X$ with:
$$\bar{X}=T\times X_1\times\cdots X_k,$$
where $T$ is a torus and the $X_i$ are irreducible symplectic orbifolds.

Indeed, let $\sigma$ be the non-degenerate holomorphic 2-form on $X^*$, then $\gamma^*(\sigma)$ is also non-degenerate. This is why there is no Calabi--Yau factors in the decomposition. 
%if the decomposition of $\bar{X}$ admitted some  factors, the form $\gamma^*(\sigma)$ would be degenerate.
\end{rmk}
\begin{rmk}
Note that in the case of a projective orbifold $X$ with $\codim\Sing X\geq4$, our definition of irreducible symplectic coincides with the one of Greb--Kebekus--Peternell \cite[Definition 8.16]{Greb} (see also \cite[Theorem 1.5]{Horing} for the Bogomolov decomposition theorem in their context). Indeed, let $X$ be an irreducible symplectic orbifold in our sense and $\gamma:Y\rightarrow X$ a quasi-étale cover. By purity of branch loci, $\gamma$ is étale over the smooth locus $X^*$ of $X$. Since $\pi_1(X^*)=0$, the cover $\gamma$ is necessarily trivial. Hence, $X$ is irreducible symplectic in the sense of Greb--Kebekus--Peternell. Reciprocally, by uniqueness of the Bogomolov decompositions, a projective orbifold which is irreducible in the sense of Greb--Kebekus--Peternell will be irreducible symplectic in our sense.
\end{rmk}

\subsection{Examples of irreducible symplectic orbifolds}
There are many of examples of primitively symplectic orbifolds (see \cite[Section 13]{Fujiki}, \cite{Markou}, \cite{Matteini}, \cite{Lol2}). However, it was not studied in the literature if these examples were irreducible symplectic.
% Moreover, irreducible symplectic orbifolds are of particular interest since they appear as elementary bricks in the generalization of the Bogomolov decomposition theorem by Campana \cite[Théorème 6.4]{Campana}.
Since we have seen in the previous section that irreducible symplectic orbifolds are of particular interest, we provide here two examples. 
%of irreducible symplectic orbifolds. 

Let $(S,i)$ be a K3 surface or a 2-dimensional complex torus endowed with a symplectic involution. When $S$ is a K3 surface (resp. 2-dimensional complex torus), the involution $i$ gives an involution $i^{[2]}$ on $S^{[2]}$ (resp. on the generalized Kummer $K_2(S)$) which fixes a K3 surface $\Sigma$ and some isolated points (see \cite[Theorem 4.1]{Mong}, (resp. \cite[Section 1.2.1]{Tari})). Since $\Fix i^{[2]}$ has connected components of codimension at least 2, the singular locus of $S^{[2]}/i^{[2]}$ (resp. $K_2(S)/i^{[2]}$) is given by the image of $\Fix i^{[2]}$ by the quotient map (see for instance \cite[Proposition 6]{Prill}). We denote by $M'$ (resp. $K'$) the blow-up of $S^{[2]}/i^{[2]}$ (resp. $K_2(S)/i^{[2]}$) in the image of $\Sigma$ when $S$ is a K3 surface (resp. a 2-dimensional complex torus). 
\begin{prop}
The orbifolds $M'$ and $K'$ are irreducible symplectic.
\end{prop}
\begin{proof}
%Since the proof is almost the same for $M'$ and $K'$, we denote these orbifolds indifferently by a common letter $X'$. 
The proof is identical for $K'$ and $M'$, so we only give the proof for $M'$.
The orbifold $M'$ is primitively symplectic because of \cite[Proposition 2.9]{Fujiki}. We denote by $U$ the smooth locus of $S^{[2]}/i^{[2]}$; we have $U=\frac{S^{[2]}\smallsetminus \Fix i^{[2]}}{i^{[2]}}$. We also denote by $U'$ the smooth locus of $M'$. We show that $\pi_1(U')=0$.

Since $S^{[2]}$ is smooth and $\codim\Fix i^{[2]}=2$, we have $\pi_1(S^{[2]}\smallsetminus \Fix i^{[2]})=0$. It follows that $\pi_1(U)=\Z/2\Z$. Moreover, we have $U=U'\smallsetminus \widetilde{\Sigma}$ where $\widetilde{\Sigma}$ is the exceptional divisor of 
$M'\rightarrow S^{[2]}/i^{[2]}$. Since $\widetilde{\Sigma}$ is simply connected, considering a tubular neighborhood of $\widetilde{\Sigma}$ in $U'$, the Seifert--van Kampen theorem provides a surjection $\pi_1(U)\rightarrow \pi_1(U')$. So $\pi_1(U')$ can only be $\Z/2\Z$ or 0.
Let $\widetilde{M}$ be the blow-up of $M'$ in the singularities. Looking at the following exact sequence:
$$\xymatrix@R10pt{
H^{2}(\widetilde{M},U',\Z)\ar[r]& H^2(\widetilde{M},\Z) \ar[r]& H^{2}(U',\Z)\ar[r]& 0},$$
the proof of \cite[Lemma 2.33]{Lol} (see \cite[Lemma 8.13 (iv)]{Kapfer} for $K'$) shows that $H^{2}(U',\Z)$ is torsion free. 
It follows that $\pi_1(U')=0$. Indeed, if we had $\pi_1(U')=\Z/2\Z$, the universal coefficient theorem would have provided $\tors H^{2}(U',\Z)=\tors H_1(U')=\Z/2\Z$ (where $\tors$ refers to the torsion).
%By the universal coefficient theorem, $0=\tors H^{2}(U',\Z)=\tors H_1(U')$, where . If we had $\pi_1(U')=\Z/2\Z$, we would have $\tors H_1(U')$
%This concludes the proof.
\end{proof}
\subsection{Stability under deformation}
%\section{Compact Kähler irreducible symplectic orbifold with $\codim\Sing X\geq4$}
%\section{Holomorphically symplectic complex analytic space}
%\subsection{Reminders on Namikawa results}
%\subsection{Kähler orbifolds}
In this section, we want to show that being irreducible or primitively symplectic orbifold is stable under deformation. In particular, it will be an object with promising properties for the construction of a moduli space (cf. Section \ref{TorelliSection}). 
%\subsection{Kähler orbifolds}
%\subsection{Deformation of orbifolds}
First, we provide some reminders from \cite[Section 3]{Fujiki}.
\begin{defi}
Let $X$ be a compact orbifold.
A deformation $f:\mathscr{X}\rightarrow S$ of $X$ is said to be of \emph{fixed local analytic type} if any point $x\in\mathscr{X}$ admits a neighborhood $\mathscr{U}$ which is isomorphic over $f(\mathscr{U})$ to $\left(\mathscr{U}\cap f^{-1}(x)\right)\times f(\mathscr{U})$.
\end{defi}
\begin{prop}[\cite{Fujiki}, Lemma 3.3]\label{triv}
Let $X$ be a compact orbifold.
Let $f:\mathscr{X}\rightarrow S$ be a deformation of $X$. Assume that $\codim \Sing X\geq 3$, then $f$ is of fixed local analytic type.
\end{prop}
\begin{rmk}
This result has been generalized by Namikawa for $\Q$-factorial projective varieties with terminal singularities in \cite{Namikawa2}.
\end{rmk}
\begin{cor}[\cite{Fujiki}, Lemma 3.1]\label{constantsheaf}
Let $X$ be a compact orbifold.
Let $f:\mathscr{X}\rightarrow S$ be a deformation of $X$. Assume that $\codim \Sing X\geq 3$, then there exists a homeomophism: $h:\mathscr{X}\rightarrow X\times S$. In particular, 
$\DR^if_*\mathbb{K}$ is a constant sheaf on $S$ for any $i$ where $\mathbb{K}=\mathbb{R}$ or $\C$.
\end{cor}
\begin{rmk}[\cite{Fujiki}, Remark 3.4]\label{Kuranishi}
Let $X$ be a primitively symplectic orbifold. We can construct $f:\mathscr{X}\rightarrow \Def(X)$ the Kuranishi deformation of $X$ and we have: $T_0 \Def(X)\simeq H^{1,1}(X)$. 
\end{rmk}
\begin{defi}[\cite{Bingener}]
Let $\mathscr{X}\rightarrow S$ be a map of complex spaces. The map $f$ is said \emph{weakly Kähler} if there exists an open cover $(\mathscr{U}_i)$ of $\mathscr{X}$ and a collection of $\mathscr{C}^\infty$ functions $(f_i)$ each defined on $\mathscr{U}_i$ such that for all fibers $\mathscr{X}_x$ of $\mathscr{X}\rightarrow S$ the restrictions to $\mathscr{U}_i\cap\mathscr{X}_x$ of the functions $f_i$ are strictly plurisubharmonic and the restrictions to $\mathscr{U}_i\cap\mathscr{U}_i\cap\mathscr{X}_x$ of the functions $f_i-f_j$ are pluriharmonic.
\end{defi}
The following proposition follows from \cite[Theorem 6.3]{Bingener} and \cite[Proposition B.2.10]{Tim}.
\begin{prop}\label{weakK}
Let $X$ be a compact Kähler orbifold with $\codim \Sing X\geq 3$.
Let $f:\mathscr{X}\rightarrow S$ be a deformation of $X$. Then there exists a neighborhood $U$ of $o:=f(X)$ such that $f:\mathscr{X}_{|U}\rightarrow U$ is weakly Kähler, where $\mathscr{X}_{|U}:=f^{-1}(U)$.
\end{prop}
Now, we can provide the following proposition. 
\begin{prop}\label{defor}
Let $X$ be an irreducible (resp. primitively) symplectic orbifold. Let $f:\mathscr{X}\rightarrow \Def(X)$ be the Kuranishi deformation of $X$. We denote by $o\in \Def(X)$ the point such that $X\simeq f^{-1}(o)$.
There exists a smooth neighborhood $U$ of $o$ in $\Def(X)$ such that for all 
$t\in U$, $\mathscr{X}_t$ is an irreducible (resp. primitively) symplectic orbifold.
\end{prop}
\begin{proof}
From Proposition \ref{triv}, $f:\mathscr{X}\rightarrow \Def(X)$ is of fixed local analytic type. Hence for all $t\in \Def(X)$, $\mathscr{X}_t$ is an orbifold with 
$\codim\Sing \mathscr{X}_t\geq4$. Moreover, it follows from \cite[Lemma 4.2]{Fujiki} that $\mathscr{X}_t$ is irreducible (resp. primitively) symplectic for all $t\in \Def(X)$. By Proposition \ref{weakK}, there exists a neighborhood $U$ of $o$ such that for all $t\in U$, $\mathscr{X}_t$ is Kähler.
Then we conclude with \cite[Theorem 3.3.18]{Tim} that $U$ is smooth (see also \cite[Theorem 2.5]{Nami}).
%By Lemma 3.3 of \cite{Fujiki}, $\mathscr{X}\rightarrow S$ is of fixed analytic type. That is for all $x\in \mathscr{X}$, there is a neighborhood $U$ which is isomorphic over $f(U)$ to $U_{f(x)}\times f(U)$, where $U_{f(x)}=U\cap f^{-1}(f(x))$. It follows that $t\mapsto\codim\Sing\mathscr{X}_t$ is a constant map. Moreover, for all $t\in S$, $\mathscr{X}_t$ is an orbifold. Since $X$ is Kähler, by proposition 1.9 of \cite{Fujiki}, $X$ verifies the Hodge decomposition. It follows from Lemma 4.2 of \cite{Fujiki} that $\mathscr{X}_t$ is irreducible symplectic for all $t\in S$. By Theorem 1.1 of \cite{Greb}, the quotient singularities are rational. Then by Corollary 3.3.6 of \cite{Tim}, $\mathscr{X}_t$ are Kähler for all $t$ in a neighborhood of $o$.
\end{proof}
\subsection{Local Torelli theorem and Fujiki formula}\label{TorelliSection}
%We recall that a compact Kähler orbifolds is said \emph{irreducible symplectic} if $W:=X\smallsetminus \Sing X$ is simply connected and admits a unique, up to scalar, nondegenerate holomorphic 2-form.
We give a proof of the local Torelli theorem and the Fujiki relation in the framework of primitively symplectic orbifolds. 
I tried to provide the simplest possible proof with a full statement of the properties of the Beauville--Bogomolov form. In particular, we can remark that proving the local Torelli theorem and the Fujiki formula together brings some simplifications. This proof can be seen more as a survey of several techniques since none of its ideas are really new; they can be found in Fujiki \cite[Section 3]{Fujiki}, Beauville \cite[Theorem 5]{Beauville}, Bogomolov \cite[Lemma 1.9]{Bogo2}, Matsushita \cite[Proposition 4.1]{Mat} and Kirschner \cite[Section 3.4]{Tim}. For a more general framework see \cite{Nanikawa} and \cite[Section 3.4 and 3.5]{Tim}.

Let $X$ be a primitively symplectic orbifold of dimension $2n$.
Let $\sigma\in H^0(X,\Omega_X^2)$ with $\int_X(\sigma\overline{\sigma})^{n}=1$. Using de Rham theorem, we can define a quadratic form on $H^{2}(X,\mathbb{C})$:
$$q_X(\alpha):=\frac{n}{2}\int_X(\sigma\overline{\sigma})^{n-1}\alpha^2+(1-n)\left(\int_X\sigma^{n-1}\overline{\sigma}^{n}\alpha\right)\cdot\left(\int_X\sigma^{n}\overline{\sigma}^{n-1}\alpha\right).$$
%From Theorem \ref{Hodge}, it is equivalent to a definition using a resolution:
%$$q(\alpha):=\frac{n}{2}\int_{\widetilde{X}}(r^*(\sigma)\overline{r^{*}(\sigma)})^{n-1}r^{*}(\alpha)^2+(1-n)\left(\int_{\widetilde{X}}r^{*}(\sigma)^{n-1}\overline{r^*(\sigma)}^{n}r^{*}(\alpha)\right)\cdot\left(\int_{\widetilde{X}}r^{*}(\sigma)^{n}\overline{f^{*}(\sigma)}^{n-1}r^{*}(\alpha)\right).$$
Let $f: \mathscr{X}\rightarrow U$ be the deformation of $X$ from Proposition \ref{defor}, where $U$ is an open subset of $\Def(X)$ containing $o$ and such that all fibers are primitively symplectic orbifolds. 
By Corollary \ref{constantsheaf}, we have a canonical isomorphism which commutes with the cup-product: $u_{s}:H^{*}(X_{s},\mathbb{C})\rightarrow H^{*}(X,\mathbb{C})$ for any $s\in U$. Then we can define the \emph{period map}:
$$
\xymatrix@R0pt{
p:U\ar[r]& \mathbb{P}(H^{2}(X,\C))\\
\ \ \ \ s\ar[r] &u_{s}(\sigma_{s}),
}$$
where $\sigma_{s}$ is the symplectic holomorphic 2-form on $X_{s}$.
\begin{thm}\label{LocalTorelli}
Let $X$ be a primitively symplectic orbifold of dimension $2n$.
%and a resolution $\widetilde{X}\rightarrow X$ where $\widetilde{X}$ is Kähler.
%\begin{itemize}
%\item[]

\hspace{-0.5cm}$\bullet$\textbf{Fujiki Formula:}

\hspace{-0.5cm}There exists a unique indivisible bilinear integral symmetric form $B_{X}\in S^{2}(H^2(X,\Z)/\tors)^{\vee}$ and a unique positive constant $c_{X}\in \mathbb{Q}^*_+$, such that: 
\begin{itemize}
\item[(i)]
for any $\alpha\in H^2(X,\mathbb{C})$
$$\int_X\alpha^{2n}=c_{X}B_{X}(\alpha,\alpha)^n,$$
\item[(ii)]
and for $0\neq \sigma\in H^{0}(\Omega_{X}^{2})$
$$B_{X}(\sigma+\overline{\sigma},\sigma+\overline{\sigma})>0.$$
\end{itemize}
Moreover the signature of $B_{X}$ is $(3,b_2(X)-3)$.
%\item[]

\hspace{-0.5cm}$\bullet$\textbf{Local Torelli theorem:}

\hspace{-0.5cm}Let $\mathcal{D}:=\left\{\left.x\in\mathbb{P}(H^{2}(X,\C))\right|\ B_X(x,x)=0\ \text{and}\ B_X(x,\overline{x})>0\right\}.$
The period map $p:U\rightarrow \mathcal{D}$ is a local isomorphism.
%\end{itemize}
\end{thm}
%\begin{proof}
\begin{defi}
The bilinear form $B_X$ of Theorem \ref{LocalTorelli}, is called the Beauville--Bogomolov form of $X$.
\end{defi}
This section is dedicated to the proof of this theorem. 
We denote by $Q$ the quadric on $\mathbb{P}(H^{2}(X,\mathbb{C}))$ defined by $q_X$. The form $B_X$ will be the bilinear form associated to the quadratic form $q_X$ after multiplication by a scalar such that $B_X$ is integral and primitive.
\subsubsection*{Step 1: $p:U\rightarrow Q$ is a local isomorphism}
The foundation stone was provided by Fujiki in \cite{Fujiki} where he shows that $p:U\rightarrow \mathbb{P}(H^{2}(X,\C))$ is a local embedding studying its differential \cite[Lemma 4.3]{Fujiki}. 

%We normalize $\varphi_{s}$ to have $\int_{X_{s}}(\varphi_{s}\overline{\varphi_{s}})^{n}=1$ and we denote by $\varphi_{0}=\varphi$.
Now, we follow the argument of Beauville in the beginning of the proof of \cite[Theorem 5]{Beauville}.
Let $\alpha\in H^{2}(X,\mathbb{C})$; we write $\alpha= a\sigma+\omega+b\overline{\sigma}$, with $\omega\in H^{1,1}(X)$ and $a,b\in \mathbb{C}$. We have
\begin{equation}
q_X(\alpha)=ab+\frac{n}{2}\int_X(\sigma\overline{\sigma})^{n-1}\omega^2.
\label{cacapuant}
\end{equation}
On the other hand, we calcule the component of type $(2n,2)$ of $\alpha^{n+1}$. We find
$$(\alpha^{n+1})_{2n,2}=(n+1)(a\sigma)^{n}\cdot b\overline{\sigma}+ {n+1\choose 2}(a\sigma)^{n-1}\omega^{2},$$
then 
\begin{align*}
\int_X\alpha^{n+1}\overline{\sigma}^{n-1}
&=(n+1)a^{n-1}\left(ab+\frac{n}{2}\int_X(\sigma\overline{\sigma})^{n-1}\omega^2\right)\\
&=(n+1)q(\alpha)\left(\int_X\alpha\sigma^{n-1}\overline{\sigma}^{n}\right)^{n-1}.
\end{align*}
Let $\sigma_{s}$ be the holomorphic 2-form on $\mathscr{X}_s$ with $\int_{\mathscr{X}_s}(\sigma_{s}\overline{\sigma_{s}})^{n}=1$, we also denote $\sigma_{o}=\sigma$.
Let $\alpha_{s}=u_{s}(\sigma_{s})$ for $s\in U$. We have 
\begin{equation}
\alpha_{s}^{n+1}=0, 
\label{InTorelli}
\end{equation}
because $\sigma_{s}$ is of type $(2,0)$ on $\mathscr{X}_{s}$. For $s$ near enough to $0$, we have $\int_X\alpha_{s}\sigma^{n-1}\overline{\sigma}^{n}\neq 0$, so $q(\alpha_{s})=0$. Then this equality remains true on all $U$. 
So $p(U)\subset Q$. So with a dimension argument (cf. Remark \ref{Kuranishi}), we conclude that $p:U\rightarrow Q$ is a local isomorphism.
\subsubsection*{Step 2: there exists $C\in\Q_+^*$ such that  $\alpha^{2n}=Cq(\alpha)^n$ for all $\alpha\in H^{2}(X,\C)$}
Let $\omega\in H^{1,1}(X)$ be a Kähler class. The class $\sigma^{n-1}$ is of type $(2n-2,0)$ so in particular is primitive. 
It follows from (\ref{cacapuant}) and Proposition \ref{Hodge--Riemann} that: 
\begin{equation}
q_X(\omega)=\frac{n}{2}\int_X(\sigma\overline{\sigma})^{n-1}\omega^2>0.
\label{Kpositive}
\end{equation}
It follows that $Q$ has rank at least 3 and so is irreducible. 
Let denote $W:=\left\{\alpha\in H^{2}(X,\mathbb{C}) | \alpha^{2n}=0\right\}$. From (\ref{InTorelli}), we have $p:U\rightarrow W$, as before it is a local embedding and because of a dimension argument, it is a local isomorphism.
Since $Q$ is irreducible, it follows that $Q\subset W$.
Actually, we can prove that $Q=W$ using the following argument from Bogomolov \cite[Lemma 1.9]{Bogo2}.
\begin{lemme}
Let $\alpha\in H^{2}(X,\mathbb{C})$ with $\int_X\alpha^{2n}=0$. Then $\alpha\in Q$.
\end{lemme}
\begin{proof}
Assume that $\int_{X}\alpha^{2n}=0$; let $l\in H^{2}(X,\mathbb{C})$, $\int_{X}l^{2n}\neq0$ and $l$ generic. The latter means that the plane $V$ generated by $l$, and $\alpha$ intersects the quadric cone $Q$ in $H^{2}(X,\mathbb{C})$ in two lines generated by $x,z\in H^{2}(X,\mathbb{C})$ respectively. The vectors $x,z$ also generate $V$.

We have $x^{n+1}=z^{n+1}=0$ and therefore $(ax+bz)^{2n}={2n\choose n}a^nb^n(x^nz^n)$, $a,b\in \mathbb{C}$. Since $l^{2n}\neq 0$ we conclude that $x^nz^n\neq0$ and $(ax+bz)^{2n}=0$ only if either $a=0$ or $b=0$. Hence $\alpha\in Q$.

\end{proof}
Since $q_X^n$ and $\int_{X}\alpha^{2n}$ have the same degree, there exists a constant $C$ such that: 
\begin{equation}
\int_X\alpha^{2n}=Cq_X(\alpha)^n,
\label{FujikiEq}
\end{equation}
for all $\alpha\in H^2(X,\C)$. Moreover, applying (\ref{FujikiEq}) to $\sigma+\overline{\sigma}$, we find that $C$ is a positive rational number.
The Fujiki formula (\ref{FujikiEq}) can be written in its polarized form:
\begin{equation}
\int_X\alpha_{1}\cdot...\cdot\alpha_{2n}=\frac{C}{(2n)!}\sum_{\sigma\in S_{2n}}q_{X}(\alpha_{\sigma(1)},\alpha_{\sigma(2)})...q_{X}(\alpha_{\sigma(2n-1)},\alpha_{\sigma(2n)}),
\label{pola}
\end{equation}
for all $\alpha_{i}\in H^{2}(X,\C)$.
\subsubsection*{Step 3: Proof of the local Torelli theorem}
For all $\alpha\in p(U)$, we have $\int_X\alpha^n\overline{\alpha}^n>0$. From (\ref{pola}), we obtain that $q_X(\alpha+\overline{\alpha})\neq0$ for all $\alpha\in p(U)$. Since $q_X(\sigma+\overline{\sigma})>0$, we get that $p(U)\subset \mathcal{D}$. So $p:U\rightarrow \mathcal{D}$ is a local isomorphism.
\subsubsection*{Step 4: Proof of the properties of $B_X$}
Now, we prove that the signature of $q_X$ is $(3,b_2(X)-3)$. 
It remains to show that for all $\alpha\in H^{1,1}_{\mathbb{R}}(X):=H^{1,1}(X)\cap H^2(X,\R)$ such that $q_X(\alpha,\omega_X)=0$, we have 
\begin{equation}
q_X(\alpha)\leq0
\label{Negative}
\end{equation}
with equality if and only if $\alpha=0$. 
From (\ref{pola}), we have:
\begin{equation}
\int_X\alpha^2\omega_X^{2n-2}=\frac{C}{2n}q_X(\alpha,\alpha)q_X(\omega_X,\omega_X)^{n-1},
\label{FujikiInter}
\end{equation}
$$\int_X\alpha\omega_X^{2n-1}=Cq_X(\alpha,\omega_X)q_X(\omega_X,\omega_X)^{n-1}=0.$$
Hence $\alpha\in H^{1,1}(X)_p$ and
by Proposition \ref{Hodge--Riemann}, we know that $\int_X\alpha^2\omega_X^{2n-2}\leq0$ with equality if and only if $\alpha=0$. So (\ref{Negative}) follows form (\ref{Kpositive}) and (\ref{FujikiInter}).

It remains to prove that there exists a constant $t>0$ such that $tq_X$ is integral and primitive on $H^2(X,\Z)$. It is enough to show that there exists a constant $t>0$ such that $tq_X$ is rational on $H^2(X,\Q)$. Let $\lambda\in H^2(X,\Q)$ such that $q_X(\lambda)>0$.
For all $\alpha\in H^2(X,\Q)$ such that $q_X(\lambda,\alpha)=0$, we have by (\ref{pola}): 
$$\int_X\lambda^{2(n-1)}\alpha^2=\frac{C}{2n}q_X(\lambda)^{(n-1)}q_X(\alpha).$$
Since $\int_X\lambda^{2(n-1)}\alpha^2\in \Q$ and $\frac{C}{2n}\in\Q$, taking $t=q_X(\lambda)^{(n-1)}$ solves the problem.

To show the uniqueness, we take another form and another constant $B'_{X}, c'_{X}$ with the same properties.
Since $B'_{X}$ and $B_{X}$ are integral indivisible and $c_{X},c'_{X}\in\mathbb{Q}^{+}$, we have $B'_{X}=\pm(\frac{c_{X}}{c'_{X}})^{\frac{1}{n}}B_{X}$ with necessarily $(\frac{c_{X}}{c'_{X}})^{\frac{1}{n}}=1$. Finally (ii) implies $B'_{X}=B_{X}$.
%\end{proof}
%\begin{nota}
%From now, we also denote by $q_X$ the quadratic form associated to $B_X$.
%\end{nota}
\begin{rmk}\label{kählerpositive}
Note that we have also seen in this proof that for $\omega_X$ a Kähler class, we have $B_X(\omega_X,\omega_X)>0$ and for $\alpha\in H^{1,1}_{\R}(X)$ such that $B_X(\alpha,\omega_X)=0$, we have $B_X(\alpha,\alpha)\leq 0$.
\end{rmk}
\subsection{Construction of the moduli space}\label{modu}
We are now ready to construct the moduli space of marked primitively symplectic orbifolds; let $X$ be such an orbifold.
Let $\Lambda$ be a nondegenerate lattice of signature $(3,b-3)$ with $b\geq3$.
The group $H^{2}(X,\Z)$, endowed with the bilinear Beauville--Bogomolov form $B_{X}$, constitutes a lattice.
Assume that $H^{2}(X,\Z)$ is isometric to $\Lambda$.
An isometry $\varphi: H^{2}(X,\Z) \rightarrow \Lambda$ is called a \emph{marking} of $X$ and $(X,\varphi)$ is called a \emph{marked primitively symplectic orbifold}. Two marked primitively symplectic orbifolds are said isomorphic $(X,\varphi)\thicksim(X',\varphi')$ if and only if there exists an isomorphism $g:X\simeq X'$ such that $g^{*}=\varphi^{-1}\circ \varphi'$. We denote by $\mathscr{M}_{\Lambda}$ the set of isomorphism classes of marked primitively symplectic orbifolds $(X,\varphi)$ with $\varphi:H^2(X,\Z)\rightarrow\Lambda$.
%We define the moduli space $\mathscr{M}_{\Lambda}=\left\{\left.(X,\varphi)\right|\right\}/\thicksim$, where $(X,\varphi)\thicksim(X',\varphi')$ if and only if there exists an isomorphism $g:X\simeq X'$ such that $g^{*}=\varphi^{-1}\circ \varphi'$.

As in the smooth case, the local Torelli theorem endows $\mathscr{M}_{\Lambda}$ with a structure of complex manifold. 
\begin{cor}
The local Torelli Theorem allows to endow $\mathscr{M}_{\Lambda}$ with a structure of non-separated complex manifold (the period maps are the coordinate charts). Moreover, the period maps can be glued to a global holomorphic map on all $\mathscr{M}_{\Lambda}$: 
\begin{align*}
\mathscr{P}:\ & \mathscr{M}_{\Lambda}\rightarrow \mathcal{D}=\mathbb{P}\left(\left\{\sigma\in \Lambda\otimes\mathbb{C} |\ \sigma^{2}=0,\ (\sigma+\overline{\sigma})^{2}>0\right\}\right)\subset \mathbb{P}(\Lambda\otimes\mathbb{C})\\
& (X,\varphi)\mapsto \varphi(H^{2,0}(X)),
\end{align*}
which is a local isomorphism.
\end{cor}
It is a natural question to ask what are the non-separated points in $\mathscr{M}_{\Lambda}$? Generalizing Huybrechts' ideas \cite[Theorem 4.3]{Huybrechts5}, it is not too hard to show that non-separated points correspond to bimeromorphic orbifolds. 
\begin{prop}\label{separated}
Let $(X,\varphi)$ and $(X',\varphi')$ be two non-separated distinct points in $\mathscr{M}_{\Lambda}$. Then $X$ and $X'$ are bimeromorphic and $\mathscr{P}(X,\varphi)=\mathscr{P}(X',\varphi')$ is contained in $\mathcal{D}\cap \alpha^\bot$ for some $0\neq\alpha\in\Lambda$. 
\end{prop}
%The end of this section is dedicated to the proof of this proposition.
\begin{lemme}\label{resolution}
Let $X$ be a compact orbifold with $\codim \Sing X\geq 4$.
Let $f:\mathscr{X}\rightarrow S$ be a deformation of $X$ with $S$ smooth.
%with $S$ a smooth relatively compact subspace in a bigger deformation base. 
After possibly shrinking $S$, there exists a resolution $r:\widetilde{\mathscr{X}}\rightarrow \mathscr{X}$ given by a finite sequence of blow-ups such that for all $s\in S$, $r_{|\widetilde{f}^{-1}(s)}:\widetilde{\mathscr{X}}_s\rightarrow \mathscr{X}_s$ is a resolution of $\mathscr{X}_s$, where $\widetilde{f}:=f\circ r$, $\widetilde{\mathscr{X}}_s$ is the fiber of $\widetilde{f}$ over $s$ and $\mathscr{X}_s$ the fiber of $f$ over $s$.
\end{lemme}
\begin{proof}
Bierstone and Milman in \cite{BM} provide an algorithm to produce a resolution by a finite sequence of blow-ups in smooth centers. When the complex space $X$ to desingularize is embedded in a smooth manifold $M$, they introduce, in \cite[Section 1]{BM}, an invariant $\inv_X^e:X \rightarrow \Sigma$, where $\Sigma$ is a partially ordered set. This invariant allows to decide which locus to blow up in each steps of the process of desingularization. In \cite[Section 13]{BM}, Bierstone and Milman remark that the invariant $\inv_X^e$ is independent of the embedding $X\subset M$. 
This allows them to define $\inv_X^e$ when $X$ is not embedded and to state \cite[Theorem 13.2]{BM} which provides a desingularization of $X$ (not embedded) by a finite sequence of blow-ups in smooth centers. Moreover this desingularization process has a universal property that can be expressed as follows.
Let $X$ and $Y$ be two complex spaces with their sequences of blow-ups as provided by \cite[Theorem 13.2]{BM}: 
$$\xymatrix@R5pt{ \cdots\ar[r]&X_{k+1}\ar[r]^{r_{k+1}}&\ar[r]X_{k}&\cdots\ar[r]&X_1\ar[r]^{r_1}&X_0=X,\\
\cdots\ar[r]&Y_{k+1}\ar[r]^{\rho_{k+1}}&\ar[r]Y_{k}&\cdots\ar[r]&Y_1\ar[r]^{\rho_1}&Y_0=Y.
}$$
Let $U\subset X$ and $V\subset Y$ be two open sets endowed with an isomorphism $\varphi:U\simeq V$. Then $\varphi$ lifts to isomorphisms throughout the entire desingularization towers:
$$\xymatrix@R10pt{ \cdots\ar[r]&r_{k+1}^{-1}(U)\eq[d]\ar[r]^{r_{k+1}}&\ar[r]r_k^{-1}(U)\eq[d]&\cdots\ar[r]&r_1^{-1}(U)\eq[d]\ar[r]^{r_1}&U\eq[d],\\
\cdots\ar[r]&\rho_{k+1}^{-1}(V)\ar[r]^{\rho_{k+1}}&\ar[r]\rho_k^{-1}(V)&\cdots\ar[r]&\rho_1^{-1}(V)\ar[r]^{\rho_1}&V.
}$$
Now, we apply \cite[Theorem 13.2]{BM} to our case. 
After possibly shrinking $S$, we can assume that it is relatively compact.
Then by Lemma \ref{triv}, there exists a finite open cover $(\mathscr{V}_i)_{i\in I}$ of $\mathscr{X}$ such that for all $i\in I$, $\mathscr{V}_i$ is isomorphic over $f(\mathscr{V}_i)$ to $\left(\mathscr{X}_{s_i}\cap \mathscr{V}_{i}\right)\times f(\mathscr{V}_i)$ for $s_i\in f(\mathscr{V}_i)$.
For simplicity in the notation, we denote $\mathscr{V}_{i,s_i}:=\mathscr{V}_i\cap \mathscr{X}_{s_i}$. Let 
$$\xymatrix@R5pt{ \cdots\ar[r]&\mathscr{V}_{i,s_i}^{k+1}\ar[r]^{r_{k+1}}&\ar[r]\mathscr{V}_{i,s_i}^{k}&\cdots\ar[r]&\mathscr{V}_{i,s_i}^1\ar[r]^{r_1}&\mathscr{V}_{i,s_i}\\
}$$
be the sequence of blow-ups given by \cite[Theorem 13.2]{BM} for each $\mathscr{V}_{i,s_i}$ (after a finite number of blow-ups $k_i$ we obtain a desingularization of $\mathscr{V}_{i,s_i}$ and for $k>k_i$ the maps $r_k$ are isomorphisms). Since $S$ is smooth, the universal property of \cite[Theorem 13.2]{BM} allows to glue all the $\mathscr{V}_{i,s_i}^1\times f(\mathscr{V}_i)$ together to provide a blow-up $\mathscr{X}_1\rightarrow \mathscr{X}$. This gluing can be repeated in each step of the desingularization to obtain a sequence:
$$\xymatrix@R5pt{ \cdots\ar[r]&\mathscr{X}_{k+1}\ar[r]^{r_{k+1}}&\ar[r]\mathscr{X}_{k}&\cdots\ar[r]&\mathscr{X}_1\ar[r]^{r_1}&\mathscr{X},\\
}$$
which leads to a desingularization of $\mathscr{X}$ with the desired properties.
\end{proof}
\begin{proof}[Proof of Proposition \ref{separated}]
As explained in \cite[Lemma 4.1]{Huybrechts5}, using the local Torelli theorem (Theorem \ref{LocalTorelli}), we can find 1-dimensional deformations $\mathscr{X}\rightarrow S$ and $\mathscr{X}'\rightarrow S$ of $\mathscr{X}_0\simeq X$ and $\mathscr{X}'_0\simeq X'$ respectively such that there exists a non-empty open subset $V\subset S$ with $0\in\partial V$ and $\mathscr{X}_{|V}\simeq \mathscr{X}'_{|V}$. 
From Lemma \ref{resolution}, there exist resolutions of the deformations $r:\widetilde{\mathscr{X}}\rightarrow \mathscr{X}$ and $r':\widetilde{\mathscr{X}'}\rightarrow \mathscr{X}'$. Moreover, as explained in the previous proof, the resolution provided by \cite[Theorem 13.2]{BM} has a universal property. That means that we still have $\widetilde{\mathscr{X}}_{|V}\simeq \widetilde{\mathscr{X}'}_{|V}$ with a commutative diagram:
\begin{equation}
\xymatrix{\widetilde{\mathscr{X}}_{|V}\ar[d]^{r}\eq[r]&\ar[d]^{r'}\widetilde{\mathscr{X}'}_{|V}\\
\mathscr{X}_{|V}\eq[r]&\mathscr{X}'_{|V}.
}
\label{Comudefor}
\end{equation}
Hence $\widetilde{\mathscr{X}}_0\simeq \widetilde{X}$ and $\widetilde{\mathscr{X}'}_0\simeq \widetilde{X'}$ are ''non-separated manifolds''. We are going to prove that $\widetilde{X}$ and $\widetilde{X'}$ are Kähler manifolds endowed with a unique holomorphic 2-form nondegenerate on a dense open subset. 

From Lemma \ref{resolution}, the resolutions $r$ and $r'$ are given by finite sequence of blow-ups. So from Lemma \ref{Kähler}, $\widetilde{X}$ and $\widetilde{X'}$ are Kähler manifolds. Moreover, $\widetilde{X}$ and $\widetilde{X'}$ can be endowed with a holomorphic 2-form 
%$\widetilde{\sigma}$ and $\widetilde{\sigma'}$ respectively 
by pulling back the holomorphic 2-forms on $X$ and $X'$. By \cite[Section 1.7 d)]{Fujiki}, these holomorphic 2-forms are unique on $\widetilde{X}$ and $\widetilde{X'}$ respectively. 
%Let $W:=X\smallsetminus \Sing X$, we have an exact sequence and a commutative diagram given by the inclusions $W\hookrightarrow X$ and $W\hookrightarrow \widetilde{X}$:
%$$\xymatrix@R10pt{
%& &H^2(X,\C)\ar[dl]_{r_0^*}\eq[d]\\
%H^2(\widetilde{X},W,\C)\ar[r]^\psi &H^2(\widetilde{X},\C)\ar[r] &H^2(W,\C).
%    }$$
%    The isomorphism $H^2(W,\C)\simeq H^2(X,\C)$ comes from the fact that $\codim \Sing X\geq 4$ (see for instance \cite[Lemma 1.6]{Fujiki}).
%As explained in \cite[Proposition 11.20]{Voisin}, $\Ima \psi\subset H^{1,1}(\widetilde{X})$. Moreover by Theorem \ref{Hodge} (iii), $r_0^*$ is a morphism of Hodge structure. It follows that $r_{0|H^{2,0}(X)}^*:H^{2,0}(X)\rightarrow H^{2,0}(\widetilde{X})$ is an isomorphism.

Knowing that $\widetilde{X}$ and $\widetilde{X'}$ are Kähler manifolds endowed with a unique holomorphic 2-form nondegenerate on a dense open subset, we can apply the technique of the proof of \cite[Theorem 4.3]{Huybrechts5} to $\widetilde{X}$ and $\widetilde{X'}$. Let $t_i$ be a sequence in $V$ converging to $0$. As in \cite[Theorem 4.3]{Huybrechts5}, the graphs $\widetilde{\Gamma}_{i}$ of the isomorphism $\widetilde{g_i}:\widetilde{\mathscr{X}}_{t_i}\rightarrow \widetilde{\mathscr{X}'}_{t_i}$ will converge to a cycle $$\widetilde{\Gamma}=\widetilde{Z}+\sum \widetilde{Y}_k,$$ where the component $\widetilde{Z}$ defines a bimeromorphic correspondence $\widetilde{g}:\widetilde{X}\dashrightarrow \widetilde{X}'$ and the components $\widetilde{Y}_k$ do not dominate neither of the two factors. 

Now, we consider $$\Gamma:=r_0\times r_0'(\widetilde{\Gamma})=r_0\times r_0'(\widetilde{Z})+\sum r_0\times r_0'(\widetilde{Y}_k).$$ 
For simplicity, we denote $Z:=r_0\times r_0'(\widetilde{Z})$ and $Y_k:=r_0\times r_0'(\widetilde{Y}_k)$ and we obtain:
\begin{equation}
\Gamma=Z+\sum Y_k.
\label{graphe}
\end{equation}
By Diagram (\ref{Comudefor}), $\Gamma$ is also the degeneration of the graphs of the isomorphisms $g_i:\mathscr{X}_{t_i}\rightarrow \mathscr{X}'_{t_i}$.
Moreover, we still have that the components $Y_k$ dominate neither $X$ nor $X'$ and $Z$ defines the graph a bimeromorphic map $g:X\dashrightarrow X'$ such that the following diagram commutes:
\begin{equation}
\xymatrix{\widetilde{X}\ar[d]^{r_0}\ar@{.>}[r]^{\widetilde{g}}&\ar[d]^{r_0'}\widetilde{X'}\\
X\ar@{.>}[r]^g&X'.
}
%\label{Comubira}
\end{equation}

We have proved that $X$ and $X'$ are bimeromorphic; now we are going to prove that $\mathscr{P}(X,\varphi)=\mathscr{P}(X',\varphi')\in\mathcal{D}\cap \alpha^\bot$ for some $0\neq\alpha\in\Lambda$. This is equivalent, by definition of the period map, to proving that $H^{1,1}(X)\cap H^{2}(X,\Z)\neq 0$ and also equivalent to having $\NS (X)$ non-trivial by Proposition \ref{Lefschetz11}. We adapt the proof of \cite[Proposition 4.7]{Huybrechts6}. We denote by $p:\Gamma\rightarrow X$ and $p':\Gamma\rightarrow X'$ the projections. 

For the first case, assume that $Z$ is not the graph of an isomorphism. In this case, $p':Z\rightarrow X'$ contracts some spaces. We can find $x\in X'$ such that $V_x=p'^{-1}(x)$ is not a point. By Proposition \ref{MildSingu}, an orbifold is always normal, hence $\dim V_x\geq 1$. We are going to prove that $V_x$ contains a curve. The projection $p'$ defines a bimeromorphic map, we can also consider: $p'^{-1}:X'\dashrightarrow Z$. Moreover by Hironaka's theorems, we can find a sequence of blow-ups $\widehat{r}:\widehat{X'}\rightarrow X'$ such that $\widehat{r}\circ p'^{-1}$ extends to a holomorphic map: $h:\widehat{X'}\rightarrow Z$. The cycle $E_x:= \widehat{r}^{-1}(x)$ is a projective variety. Since $h$ provides a holomorphic map: $h:E_x\rightarrow V_x$, by \cite[Theorem 2]{Moishezon}, $V_x=h(E_x)$ is Moishezon. 
So $V_x$ contains a curve. That means that $X$ contains a curve. Thus by Proposition \ref{typepp}, $H^{2n-1,2n-1}(X)\cap H^{2n-2}(X,\Q)\neq 0$. It follows by Theorem \ref{Hodge} (ii) and Corollary \ref{poincaréQ} that $H^{1,1}(X)\cap H^{2}(X,\Q)\neq 0$ and so $H^{1,1}(X)\cap H^{2}(X,\Z)\neq 0$.

For the second case, assume that $Z$ is the graph of an isomorphism. 
%By Diagrams (\ref{}) and (\ref{}), the action of $[Z]_*+\sum [Y_k]_*$ on 
We consider the action of $[\Gamma]_*=[Z]_*+\sum [Y_k]_*$ on $\Lambda$ (via the given markings $\varphi$ and $\varphi'$, that is $x\mapsto\varphi'(p'_*([\Gamma]\cdot p^*\varphi^{-1}(x)))$).
% for $x\in \Lambda$). 
 If the images of all the $Y_k$ are of codimension $\geq2$ in $X$ and $X'$, then $\left[Z\right]_*=\left[ \Gamma\right]_*$ on $\Lambda$ because $\varphi'(p'_*([Y_k]\cdot p^*\varphi^{-1}(x)))=0$ for all $x\in \Lambda$ and all $k$. Moreover by diagram (\ref{Comudefor}), $\Gamma$ is the degeneration of the graphs of the isomorphisms $g_i:\mathscr{X}_{t_i}\rightarrow \mathscr{X}'_{t_i}$; hence
$\left[\Gamma\right]_*=\left[\Gamma_{g_i}\right]_*$. Since $g_i$ is compatible with the markings, the action of $\left[\Gamma_{g_i}\right]_*$ is the identity. It follows that $Z$ is the graph of an isomorphism $X\simeq X'$ which is compatible with the markings $\varphi$ and $\varphi'$. That means $(X,\varphi)=(X',\varphi')$. Since $(X,\varphi)$ and $(X',\varphi')$ were assumed to be distinct, this case can be excluded. Hence one of the $Y_k$ maps onto a divisor in $X$ or in $X'$. So $H^{1,1}(X)\cap H^{2}(X,\Z)\neq 0$ or equivalently $H^{1,1}(X')\cap H^{2}(X',\Z)\neq 0$. This finishes the proof. 
\end{proof}
\begin{rmk}\label{nonsepgraph}
Let $(X,\varphi)$ and $(X',\varphi')$ be two non-separated distinct points in $\mathscr{M}_{\Lambda}$.
We have proved that $X$ and $X'$ are bimeromorphic. Actually, we have proved a bit more. From (\ref{graphe}), we have seen that we can find a cycle in $X\times X'$:
$$\Gamma=Z+\sum Y_k,$$
where $Z$ defines a bimeromorphic correspondence between $X$ and $X'$, the components $Y_k$ dominate neither $X$ nor $X'$ and $p'_*([\Gamma]\cdot p^*\alpha)=\varphi'^{-1}\circ\varphi(\alpha)$ for all $\alpha\in H^{2}(X,\Z)$.
\end{rmk}
From Proposition \ref{separated}, we can construct, with exactly the same proof as Huybrechts in \cite[Section 4.3]{Huybrechts6}, a Hausdorff reduction of $\mathscr{M}_{\Lambda}$.
\begin{cor}\label{Hausdorff}
The period map $\mathscr{P}:\mathscr{M}_{\Lambda}\rightarrow \mathcal{D}$ factorizes through the 'Hausdorff reduction' $\overline{\mathscr{M}_{\Lambda}}$ of $\mathscr{M}_{\Lambda}$. More precisely, there exists a complex Hausdorff manifold $\overline{\mathscr{M}_{\Lambda}}$ and a locally biholomorphic map factorizing the period map:
$$\mathscr{P}: \mathscr{M}_{\Lambda}\twoheadrightarrow \overline{\mathscr{M}_{\Lambda}}\rightarrow \mathcal{D},$$
such that $x=(X,\varphi), y=(X',\varphi')\in \mathscr{M}_{\Lambda}$ map to the same point in $\overline{\mathscr{M}_{\Lambda}}$ if and only if they are inseparable points of $\mathscr{M}_{\Lambda}$.
\end{cor}
%\subsection{An example of non-separated irreducible symplectic orbifolds}
\section{Projectivity criterion for irreducible symplectic orbifolds}\label{MainSection}
In this section, we prove the projectivity criterion (Theorem \ref{PC}).
The main tool of the proof is the use of currents; in the next section we provide a quick reminder about these objects.
\subsection{Reminder on currents}
Let X be a complex manifold of dimension $n$. We denote by $\mathcal{A}_c^{p,p}(X)$ the space of smooth $(p,p)$-forms with compact support on $X$. Let $\Omega\subset X$ be a coordinate open set. Let $\varphi_i$ be a collection of $(p,p)$ forms on $\Omega$ that gives a basis of the space of $(p,p)$ forms at each point of $\Omega$. Let $E\subset \Omega$ be a relatively compact set and $s\in \mathbb{N}$, for all %forms $\phi=\sum f_j(z)\varphi_i\in \mathcal{A}_c^{n-1,n-1}(X)$ and a relatively compact set $K\subset X$, we define the norm:
forms written $\phi=\sum f_j(z)\varphi_i\in \mathcal{A}^{p,p}_c(X)$ on $\Omega$, we define the semi-norm:
$$p_{\infty,E}^s(\phi)=\max_j\max_{k\leq s}\left\|f_j^{(k)}\right\|_{\infty,E}.$$
We endow $\mathcal{A}_c^{p,p}(X)$ with the topology defined by all the semi-norms $p_{\infty,E}^s$ when $E$, $s$ and $\Omega$ vary.
A form $\varphi\in \mathcal{A}^{p,p}(X,\R)$ is said to be \emph{positive} if it can be written locally in the form $i\alpha_1\wedge\overline{\alpha}_1\wedge...\wedge i\alpha_{p}\wedge\overline{\alpha}_{p}$, where the $\alpha_{i}$'s are smooth $(1,0)$-forms; if moreover $\varphi$ is nowhere 0, $\varphi$ is said \emph{strictly positive}.

A \emph{$(1,1)$-current} is a continuous linear map $T: \mathcal{A}_c^{n-1,n-1}(X)\rightarrow\C$.
%We denote by $$ the real forms; 
A $(1,1)$-current $T$ is said \emph{positive} if for all positive forms $\varphi\in \mathcal{A}_c^{n-1,n-1}(X,\R)$, we have $T(\varphi)\geq 0$. A current $T$ is said \emph{closed} if $T$ factorizes over $\mathcal{A}_c^{n-1,n-1}(X)/\mathcal{A}_c^{n-1,n-1}(X)\cap\Ima d$. A closed $(1,1)$-current $T$ is said \emph{Kähler} if there exists a strictly positive $(1,1)$-form $\omega$ such that $T-\omega$ is a positive current. Moreover we can define the de Rham cohomology of currents by setting $\dif T=T\circ \dif$. Let us denote by $H^*_{dRc}(X)$ the \emph{de Rham cohomology of currents} on $X$. It is well known that $H^*_{dRc}(X)\simeq H^*_{dR}(X)$. Then we call a current $T$ \emph{integral}, if its de Rham cohomology class is in $H^*(X,\Z)$.
%We finish with a remark about weak compactness. 
%\begin{lemme}\label{currents}
%Let $X$ be a compact Kähler manifold of dimension $n$. Let $\omega$ be a Kähler form and $(T_n)$ a sequence of positive $(1,1)$-currents bounded in %mass, that is $(T_n(\omega^{n-1}))$ is bounded. Then there exists a sub-sequence $(T_{\mu(n)})$ which weakly converges to a positive $(1,1)$-current %$T$.
%\end{lemme}
%\begin{proof}
Let us say one word about the \emph{weak compactness} of sets of currents in the following remark. 
\begin{rmk}\label{weakcompactness}
The weak topology on the space of $(1,1)$-currents $\mathcal{A}'_{1,1}(X)$ is the topology defined by the collection of semi-norms $T\rightarrow \left|T(\varphi)\right|$ for $\varphi\in \mathcal{A}_c^{n-1,n-1}(X)$. A set $B\subset \mathcal{A}'_{1,1}(X)$ is said weakly bounded if $T(\varphi)$ is bounded when $T$ runs over $B$, for every fixed $\varphi\in \mathcal{A}_c^{n-1,n-1}(X)$. The theorem of Banach-Alaoglu implies that every weakly bounded closed subset $B\subset \mathcal{A}'_{1,1}(X)$ is weakly compact.
\end{rmk}
The theorem of Banach-Alaoglu will be used in the proof of Lemma \ref{final}. For this purpose, we recall a method to 
%Moreover, there is a simple way to 
check that a set of positive $(1,1)$-currents is weakly bounded. It uses the notion of \emph{trace measure}. 
%As explained in \cite[Proposition III 1.14]{Demailly}, a positive $(1,1)$-current is of \emph{order zero}. We recall the definition of an order zero current. 
For simplicity, we assume that $X$ is compact. 
%In this case, the semi-norm $p_{\infty,X}^0$ is a norm that we denote by $\left\|\cdot\right\|_\infty$.
%%Let $\varphi_i$ be a collection of $(n-1,n-1)$ forms that gives a basis of $(n-1,n-1)$ forms at each point. Then for all 
%%%forms $\phi=\sum f_j(z)\varphi_i\in \mathcal{A}_c^{n-1,n-1}(X)$ and a relatively compact set $K\subset X$, we define the norm:
%%forms $\phi=\sum f_j(z)\varphi_i\in \mathcal{A}^{n-1,n-1}(X)$, we define the norm:
%%$$\left\|\phi\right\|_\infty=\max_j\left\|f_j\right\|_{\infty}.$$
%A current $T$ is said to be of \emph{order zero} if it is continuous with respect of the topology induced by this norm. As explained in \cite[Proposition III 1.14]{Demailly}, a positive $(1,1)$-current is of order zero. 
Let $\omega$ be a Kähler form on $X$, the \emph{trace measure} of $T$ is given by $T(\omega^{n-1})$. Let $\Omega$ be a coordinate open set, it is shown in \cite[Section III.1]{Demailly} that there exists a positive constant $C>0$ such that for all positive $(1,1)$ currents $T$ and all $\varphi\in\mathcal{A}_c^{n-1,n-1}(X)$ with support in $\Omega$:
%%\opnorm{T}\leq C T(\omega^{n-1}),
%\left|T(\varphi)\right|\leq C T(\omega^{n-1})\left\|\varphi\right\|_\infty.
$$\left|T(\varphi)\right|\leq C T(\omega^{n-1})p_{\infty,\Omega}^0(\varphi).$$
We consider a finite covering $(\Omega_i)$ by coordinate open sets and a partition of unity $(f_i)$ associated to $(\Omega_i)$. There exists a positive constant $C'>0$ such that for all positive $(1,1)$ currents $T$ and all $\varphi\in\mathcal{A}^{n-1,n-1}(X)$:
\begin{equation}
\left|T(\varphi)\right|\leq C' T(\omega^{n-1})\sum_ip_{\infty,\Omega_i}^0(f_i\varphi).
\label{currents}
\end{equation}
%For simplicity, we assume that $X$ is compact. 
%where $\opnorm{T}$ is the dual norm of $\left\|\cdot\right\|_\infty$.
%Hence there exists a sub-sequence $(T_{\mu(n)})$ which weakly converges to a limit $T_\infty$ which is a current as well since $\left\|T\infty\right\|\leq C T_\infty(\omega^{n-1})$ is finite.
%\end{proof}
\begin{rmk}\label{remarkaddedbyreferee}
Let $B\subset \mathcal{A}'_{1,1}(X)$ and $\omega$ be a Kähler form on $X$. We have seen with (\ref{currents}) that it is enough to show that $\left\{\left.T(\omega^{n-1})\right|\ T\in B\right\}$ is bounded to prove that $B$ is weakly bounded.
\end{rmk}

The key result for proving Theorem \ref{PC} is the following. 
\begin{thm}[\cite{JS}, Theorem 1.1]\label{JSth}
Let $X$ be a compact complex manifold. Then the following statements are equivalent:
\begin{itemize}
\item[(1)]
$X$ is Moishezon.
\item[(2)]
There is an integral Kähler current on $X$.
\end{itemize}
\end{thm}
\subsection{Proof of Theorem \ref{PC}}
This section is dedicated to the proof of Theorem \ref{PC}.
Throughout the section $X$ is a primitively symplectic orbifold.
%\begin{lemme}\label{resolution}
%Let $f:\mathscr{X}\rightarrow S$ be a deformation of $X$. Then there exists a projective resolution of $r:\widetilde{\mathscr{X}}\rightarrow \mathscr{X}$ such that for all $s\in S$, $r_{|\widetilde{f}^{-1}(s)}:\widetilde{\mathscr{X}}_s\rightarrow \mathscr{X}_s$ is a resolution $\mathscr{X}_s$, where $\widetilde{f}:=f\circ r$, $\widetilde{\mathscr{X}}_s$ is the fiber of $\widetilde{f}$ over $s$ and $\mathscr{X}_s$ the fiber of $f$ over $s$.
%\end{lemme}
%\begin{proof}
%By Theorem 13.2 of \cite{BM}, we can find a universal resolution $r:\widetilde{\mathscr{X}}\rightarrow \mathscr{X}$ given by a sequence of blow-ups. Moreover by Lemma 3.3 of \cite{Fujiki}, $\mathscr{X}\rightarrow S$ is of fixed analytic type, it follows the wanted property. 
%\end{proof}

Let $f:\mathscr{X}\rightarrow Def(X)$ be the Kuranishi deformation of $X$. By Proposition \ref{weakK} and Theorem \ref{LocalTorelli}, we can reduce $Def(X)$ to a neighborhood $U$ of $o:=f(X)$ such that $\mathscr{X}_{U}\rightarrow U$ is weakly Kähler and the period map $p:U\rightarrow p(U)$ is an isomorphism by the local Torelli theorem (Theorem \ref{LocalTorelli}). We also allow to possibly shrink $U$ to a smaller neighborhood of $o$ for applying Lemma \ref{resolution} and \ref{Kähler2}.
From now on, to simplify the notation, we drop the subscript $U$ of $\mathscr{X}_{U}$ when considering $\mathscr{X}_{U}\rightarrow U$.

Let $\beta\in H^{2p}(X,\R)$, we consider the set $S_\beta$ of $s\in U$ such that $\beta$ is a cohomology class of type $(p,p)$. 
Then $S_\beta$ is a closed analytic subset of $U$. Indeed if we consider the resolution $r:\widetilde{\mathscr{X}}\rightarrow\mathscr{X}$, from Theorem \ref{Hodge} (iii), $S_\beta=S_{r^*(\beta)}$, where $S_{r^*(\beta)}$ is the set of $s\in U$ such that $r^*(\beta)$ is a cohomology class of type $(p,p)$.
Because of \cite[Theorem 10.9]{Voisin}, $S_{r^*(\beta)}=S_\beta$ is a closed analytic subset of $U$.

Let $A\subset H^*(X,\Z)$ be the set of all integral classes $\beta\in H^{2p}(X,\Z)$ for $p\in\left\{1,...2n\right\}$ such that $S_\beta$ is a proper subset of $U$. Let
$$\mathcal{U}:=U\smallsetminus \cup_{\beta\in A}S_\beta.$$
Since $\mathcal{U}$ is the complement of a countable union of proper closed analytic subsets, $\mathcal{U}$ is dense in $U$. In the following a \emph{very general point} will refer to a point in $\mathcal{U}$.
%as in Lemma \ref{Dense}.
%Let $\mathscr{U}\subset U$ be the subset of $U$ given by the points in the complement of $\cup_{\beta\in A}S_\beta$ and such that if $s\in \mathscr{U}$, $\mathscr{X}_s$ does not admit any analytic subsets of odd dimension.
We denote by $\mathcal{C}_Y$ the \emph{positive cone} of a primitively symplectic orbifold $Y$, it is the connected component of $$\left\{\left.\alpha\in H^{1,1}(Y,\R)\right|\ B_Y(\alpha,\alpha)>0\right\}$$ that contains the Kähler cone. 

\begin{lemme}\label{Dense}
Let $s\in \mathcal{U}$ and $\delta\in H^{2q}(\mathscr{X}_s,\Z)$ with $q$ odd such that $\delta$ is of type $(q,q)$. Then: 
$$\delta\cdot \alpha_z^{2n-q}=0,$$
for all $z\in U$ and all $\alpha_z\in \mathcal{C}_{\mathscr{X}_z}$.
%The set $\mathscr{U}$ is dense in $U$. 
\end{lemme}
\begin{proof}
%First $U\smallsetminus \cup_{\beta\in A}S_\beta$ is dense in $U$ because $\cup_{\beta\in A}S_\beta$ is a countable union of proper closed analytic subsets. 
We fix $s\in \mathcal{U}$ and $\delta\in H^{2q}(\mathscr{X}_s,\Z)$ as in the statement of the lemma.
By definition of $\mathcal{U}$, we have $S_{\delta}=U$. Hence $\delta$ is a class of type $(q,q)$ in all $H^{2q}(\mathscr{X}_z,\C)$ for all $z\in U$. 

We consider the following polynomial $Q_\delta(t,\sigma)=(\sigma+t\overline{\sigma})^{2n-q}\cdot\delta$ with $\sigma$ running on $p(U)$ and $t\in\R$.
Since $q$ is odd and $\delta$ is of type $(q,q)$ in all $H^{2q}(\mathscr{X}_z,\C)$ for all $z\in U$, we have:
\begin{equation}
Q_\delta\equiv0.
\label{Denseequa}
\end{equation}
We assume that there exists $z\in U$ and $\alpha_z\in \mathcal{C}_{\mathscr{X}_z}$ such that 
\begin{equation}
\alpha_{z}^{2n-q}\cdot\delta\neq0,
\label{Kählerequa}
\end{equation}
and we will find a contradiction. We can normalize $\alpha_z$ such that $B_X(\alpha_{z},\alpha_{z})=1$. 
Let $\sigma_{z}$ be a holomorphic 2-form on $\mathscr{X}_z$ such that $B_X(\sigma_z,\overline{\sigma_z})=1$.
Now we consider the class 
$$\sigma_{\epsilon,z}:=\sigma_{z}-\epsilon^2\overline{\sigma_{z}}+\sqrt{2}\epsilon\alpha_{z}.$$
For $\epsilon$ small enough, we have $\sigma_{\epsilon,z}\in p(U)$. 
Now we consider the polynomial in the variables $(t,\epsilon)$: $$R_\delta(t,\epsilon)= Q_\delta(t,\sigma_{\epsilon,z}).$$
We have:
\begin{align*}
R_\delta(\epsilon^2,\epsilon)&=(\sigma_{\epsilon,z}+\epsilon^2\overline{\sigma}_{\epsilon,z})^{2n-q}\cdot\delta\\
&=\left((1-\epsilon^4)\sigma_{z}+\sqrt{2}(\epsilon+\epsilon^3)\alpha_{z}\right)^{2n-q}\cdot\delta\\
&=2^{\frac{2n-q}{2}}(\epsilon+\epsilon^3)^{2n-q}\alpha_{z}^{2n-q}\cdot\delta.
\end{align*}
Hence, by (\ref{Kählerequa}), we have $R_\delta(\epsilon^2,\epsilon)\neq0$ for $\epsilon\neq0$.
%However, since $q$ is odd and $\delta$ is of type $(q,q)$ in all $H^{2q}(\mathscr{X}_z,\C)$ for all $z\in U$, we have $Q_\delta\equiv0$ and so $R_\delta\equiv0$. This is a contradiction.
This is in contradiction with (\ref{Denseequa}).

\end{proof}
\begin{rmk}
The previous lemma could also have been proved using the twistor space (cf. Section \ref{twisty}).
\end{rmk}
%It follows that when $p$ is odd $S_\beta$ is a proper subset of $U$. Now, let $A\subset H^*(X,\Z)$ be the set of all integral classes $\beta$ such that $S_\beta$ is a proper subset of $U$.
%Then the set $U\smallsetminus \cup_{\beta\in A} S_\beta$ is dense in $U$. 
%Indeed, $U\smallsetminus \cup_{\beta\in A} S_\beta$

\begin{lemme}\label{FujikiHuybrechts}
If $\beta\in H^{4p}(X,\C)$ is of type $(2p,2p)$ on all small deformations of $X$, then there exists a constant $c_\beta$ depending on $\beta$ such that for all $\alpha\in H^{2}(X,\C)$, one has $\beta\cdot\alpha^{2(n-p)}=c_\beta B_X(\alpha,\alpha)^{n-p}$. 
\end{lemme}
\begin{proof}
In the smooth case, this is a corollary of the Local Torelli theorem.
In our case, the proof of \cite[Theorem 5.12]{Huybrechts} applies word by word using Theorem \ref{LocalTorelli}.
\end{proof}
\begin{lemme}\label{Kähler2}
Let $r:\mathscr{X}_1\rightarrow\mathscr{X}$ be a blow up in a smooth center. We denote by $E$ the exceptional divisor and by $\mathscr{X}_{1,t}$, $t\in U$ the fibers of the composition $\mathscr{X}_1\rightarrow\mathscr{X}\rightarrow U$.
Then 
%(eventually after shrinking $U$ to a smaller open neibourhood of $o$), 
there exist a family $(\omega_t)_{t\in U}$ of Kähler forms on $(\mathscr{X}_t)_{t\in U}$ depending continuously on $t$ 
%a linear combination of exceptional divisors $E$ of $\widetilde{\mathscr{X}}\rightarrow\mathscr{X}$ 
and a real number $R_0\in \mathbb{R}_+^*$ such that for all $t\in U$ and all $R\geq R_0$ 
$$R\left[r^*(\omega_t)\right]-\left[E\right]_{|\mathscr{X}_{1,t}}$$ is a Kähler class of $\mathscr{X}_{1,t}$.
\end{lemme}
\begin{proof}
We have chosen $U$ such that $\mathscr{X}\rightarrow U$ is weakly Kähler. This provides a family $(\omega_t)_{t\in U}$ of Kähler forms depending continuously on $t$. Then the lemma can be proved exactly as Lemma \ref{Kähler}.
\end{proof}

Let $r:\widetilde{\mathscr{X}}\rightarrow\mathscr{X}$ be a resolution as in Lemma \ref{resolution}. We denote $\widetilde{f}:=r\circ f$.
%The resolution of Lemma \ref{resolution} provides the following commutative diagram:
%\begin{equation}
%\xymatrix{
% \mathscr{X}\times\mathbb{P}^N \ar[dr]^{\pi_1}\ar[r]^{\pi_2} &  \mathbb{P}^N&\\
%\widetilde{\mathscr{X}}\ar[u]^{j} \ar[r]^{r}\ar@/_/[rr]_{\widetilde{f}}&\mathscr{X} \ar[r]^{f} & U,
%    }
%\label{Resolution}
%\end{equation}
%    where $\pi_1$ and $\pi_2$ are the projections, $j$ is an embedding and $N\in\mathbb{N}$. We also denote by $r_s:=r_{|\widetilde{f}^{-1}(s)}$ and by $H$ the class of a hypersurface in $\mathbb{P}^N$.
Let $(t_k)$ a sequence of points of $U$, to simplify the notation, we denote the fibers: $\mathscr{X}_{k}:=\mathscr{X}_{t_k}=f^{-1}(t_k)$, $\widetilde{\mathscr{X}}_{k}:=\widetilde{\mathscr{X}}_{t_k}=\widetilde{f}^{-1}(t_k)$ and $r_k:=r_{|\widetilde{\mathscr{X}}_{k}}$.
\begin{lemme}\label{kählergene}
Let $(t_k)$ be a sequence of points of $\mathcal{U}$ and $\alpha_k\in \mathcal{C}_{\mathscr{X}_{k}}$ a sequence of classes such that $\left\{\left.B_{\mathscr{X}_{k}}(\alpha_k,\alpha_k)\right|\ k\in\mathbb{N}\right\}$ is bounded from below. Then, there exist an integral linear combination of exceptional divisors of $r$ denoted by $E$ and $R_0\in \mathbb{R}^*_+$ such that for all $R\geq R_0$, and all $k\in\mathbb{N}$, $$R r_{k}^*(\alpha_k)+\left[E\right]_{|\widetilde{\mathscr{X}}_{k}}$$
is a Kähler class.
%Let $s\in\mathscr{U}$, $\alpha\in\mathcal{C}_{\mathscr{X}_s}$ and $t\in\mathbb{R}^*_+$, then $r_s^*(\alpha)+t(\pi_2\circ j)^*(H)_{|\widetilde{\mathscr{X}_s}}$ is a Kähler class on $\widetilde{\mathscr{X}_s}$.
\end{lemme}
\begin{proof}
The map $r$ is obtained by a finite number of blow-ups. Hence applying Lemma \ref{Kähler2} recursively with integral coefficients, we can find an integral linear combination of exceptional divisors $E$, a family $(\omega_t)_{t\in U}$ of Kähler forms depending continuously on $t$ and a real number $T_0$ such that for all $t\in U$ and all $T\geq T_0$,
$$T\left[r_t^*(\omega_t)\right]+\left[E\right]_{|\widetilde{\mathscr{X}}_{t}}$$ is a Kähler class on $\widetilde{\mathscr{X}}_{t}$.

Let $(t_k)$ be a sequence of points of $\mathcal{U}$. As previously, for simplifying the subscripts, we denote $\omega_k:=\omega_{t_k}$. 
%a sequence of Kähler forms $(\omega_k)$ and a real number $T_0$ such that for all $k\in\mathbb{N}$ and all $T>T_0$,
%$$T\left[r_k^*(\omega_k)\right]+c_1(E)_{|\widetilde{\mathscr{X}}_{k}}$$ is a Kähler class on $\widetilde{\mathscr{X}}_{k}$.
Let $Y_k$ be the class of an analytic subset of $\widetilde{\mathscr{X}}_k$ of dimension $d$ and $T\geq T_0$.
Then: $$Y_k\cdot\left(T\left[r_k^*(\omega_k)\right]+\left[E\right]_{|\widetilde{\mathscr{X}}_{k}}\right)^d>0.$$
That is:
$$\sum_{q=0}^d\binom{d}{q}T^q\left[r_k^*(\omega_k)\right]^q\cdot\left(Y_k\cdot \left[E\right]_{|\widetilde{\mathscr{X}}_{k}}^{d-q}\right)>0.$$
Taking the image by the push forward map $r_{k*}$ and using the projection formula (Remark \ref{projformula}), we obtain:
\begin{equation}
\sum_{q=0}^d\binom{d}{q}T^q\left[\omega_k\right]^q\cdot r_{k*}\left(Y_k\cdot \left[E\right]_{|\widetilde{\mathscr{X}}_{k}}^{d-q}\right)>0.
\label{lastcorrect}
\end{equation}
%The terms $r_{k*}\left(Y_k\cdot \left[E\right]_{|\widetilde{\mathscr{X}}_{k}}^{d-q}\right)$ are linear combinations of classes of analytic subsets $W_k$ of $\mathscr{X}_k$. 
Since $t_k\in \mathcal{U}$, by Lemma \ref{Dense}, we have:
\begin{equation}
 \left[\omega_k\right]^q\cdot r_{k*}\left(Y_k\cdot \left[E\right]_{|\widetilde{\mathscr{X}}_{k}}^{d-q}\right)=
 \left[\alpha_k\right]^q\cdot r_{k*}\left(Y_k\cdot \left[E\right]_{|\widetilde{\mathscr{X}}_{k}}^{d-q}\right)=0,
 \label{newkey}
\end{equation}
for all $q$ odd.
Now, we consider $q$ even.
Only the intersections of $\left[\omega_k\right]^q$ or $\left[\alpha_k\right]^q$ with classes of type $(2n-q,2n-q)$ are possibly non trivial.
%the analytic subsets $W_k$ are of even dimension $2l$.
Let $x_{k,q}$ be the component of $r_{k*}\left(Y_k\cdot \left[E\right]_{|\widetilde{\mathscr{X}}_{k}}^{d-q}\right)$ of type $(2n-q,2n-q)$. 
Since $t_k\in\mathcal{U}$, $x_{k,q}$ remains of type $(2n-q,2n-q)$ on all small deformations of $\mathscr{X}_k$.
Hence by Lemma \ref{FujikiHuybrechts}, there exists a constant $c(x_{k,q})$ depending on $x_{k,q}$ such that for all $\gamma\in H^{2}(\mathscr{X}_k,\C)$:
%\begin{equation}
$$x_{k,q}\cdot\gamma^{q}=c(x_{k,q}) B_{\mathscr{X}_{k,q}}(\gamma,\gamma)^{\frac{q}{2}}.
$$
%\label{FujikiHuybrechtsequa}
%\end{equation}
%Taking $\gamma$ to be a Kähler class in the previous equation, we deduce from Remark \ref{kählerpositive} that $c(x_k)>0$.
Applying this equation for $\gamma= \omega_k$ and for $\gamma=\alpha$, we obtain:
$$x_{k,q}\cdot\left[\omega_k\right]^{q}=\frac{B_{\mathscr{X}_k}(\left[\omega_k\right],\left[\omega_k\right])^{\frac{q}{2}}}{B_{\mathscr{X}_k}(\alpha_k,\alpha_k)^{\frac{q}{2}}}x_{k,q}\cdot\alpha_k^{q}.$$
For this reason and because of (\ref{newkey}), we can rewrite (\ref{lastcorrect}) as follows:
$$\sum_{q=0}^d\binom{d}{q}\left(\sqrt{\frac{B_{\mathscr{X}_k}(\left[\omega_k\right],\left[\omega_k\right])}{B_{\mathscr{X}_k}(\alpha_k,\alpha_k)}}T\right)^q\alpha_k^q\cdot r_{k*}\left(Y_k\cdot \left[E\right]_{|\widetilde{\mathscr{X}}_{k}}^{d-q}\right)>0.$$
%$$\sum_{q'=0}^\left\lfloor d\right\rfloor\binom{d}{2q'}T_0^{2q'}\alpha_k^{2q'}\cdot r_{k*}\left(Y_k\cdot c_1(E)_{|\widetilde{\mathscr{X}}_{k}}^{d-q}\right)>0.$$
Let $R\geq T_0\sqrt{\frac{B_{\mathscr{X}_k}(\left[\omega_k\right],\left[\omega_k\right])}{B_{\mathscr{X}_k}(\alpha_k,\alpha_k)}}$, 
then:
$$\sum_{q=0}^d\binom{d}{q}R^q\alpha_k^q\cdot r_{k*}\left(Y_k\cdot \left[E\right]_{|\widetilde{\mathscr{X}}_{k}}^{d-q}\right)>0.$$
That is: 
$$Y_k\cdot\left(Rr_k^*(\alpha_k)+\left[E\right]_{|\widetilde{\mathscr{X}}_{k}}\right)^d>0.$$
Using the criterion of Demailly and Paun \cite[Theorem 0.1]{DP}, this proves that $Rr_k^*(\alpha_k)+\left[E\right]_{|\widetilde{\mathscr{X}}_{k}}$ is a Kähler class on $\widetilde{\mathscr{X}}_k$.
Moreover, the map $t\rightarrow B_{\mathscr{X}_t}(\left[\omega_t\right],\left[\omega_t\right])$ is continuous and so is bounded on $\overline{\left\{t_k\left|k\in\mathbb{N}\right.\right\}}$. Since we assumed $(B_{\mathscr{X}_k}(\alpha_k,\alpha_k))_{k\in\mathbb{N}}$ bounded from below,
we can find $R_0$ such that:
$$
R_0\geq T_0\sqrt{\frac{B_{\mathscr{X}_k}(\left[\omega_k\right],\left[\omega_k\right])}{B_{\mathscr{X}_k}(\alpha_k,\alpha_k)}},$$
for all $k\in\mathbb{N}$. This finishes the proof.

\end{proof}
%\begin{rmk}\label{KählerRemark}
%With the same technique, we can also remark that if $s\in U$, $\omega_s$ is Kähler class on $\mathscr{X}_s$ and $t\in\mathbb{R}^*_+$, then $r_s^*(\alpha)+t(\pi_2\circ j)^*(H)_{|\widetilde{\mathscr{X}_s}}$ is a Kähler class of $\widetilde{\mathscr{X}_s}$.
%\end{rmk}
\begin{lemme}\label{final}
Let $s\in U$ and $\alpha\in\mathcal{C}_{\mathscr{X}_s}$, then there exist an integral linear combination of exceptional divisors $E$ and $R_0\in\R_+^*$ such that 
$\widetilde{\alpha}_{R}:=Rr_s^*(\alpha)+\left[E\right]_{|\widetilde{\mathscr{X}}_{s}}$ is the class of a Kähler current for all $R\geq R_0$.
%is in the interior of the pseudo-effective cone $\mathring{H}_{psef}^{1,1}(\widetilde{\mathscr{X}_s})$.
\end{lemme}
\begin{proof}
We adapt the proof of \cite[Proposition 1]{Huybrechts2}. 
By definition, $\mathcal{U}$ is dense in $U$. Hence there exists $(s_i)_{i\in\mathbb{N}}$ a sequence in $\mathcal{U}$ which converges to $s$.
Moreover, $\alpha$ can be approximated by a sequence $(\alpha_{i})\in H^2(\mathscr{X}_s,\Z)$ such that $\alpha_i\in\mathcal{C}_{\mathscr{X}_{s_i}}$.
By Lemma \ref{kählergene}, there exists $R_0\in\R_+^*$ and an integral combination of exceptional divisors $E$ of $r$ such that for all $i\in\mathbb{N}$ and all $R\geq R_0$,
$\widetilde{\alpha}_{i,R}:=Rr_{s_i}^*(\alpha_i)+\left[E\right]_{|\widetilde{\mathscr{X}}_{s_i}}$ 
is a Kähler class and therefore corresponds to a closed positive $(1,1)$ current on $\widetilde{\mathscr{X}_{s_i}}$. 
%We denote $\widetilde{\alpha}_i:=Tr_{s_i}^*(\alpha_i)+(\pi_2\circ j)^*(H)_{|\widetilde{\mathscr{X}_{s_i}}}$.
Remark that all the currents $\widetilde{\alpha}_{i,R}$ can be seen as currents (not necessarily (1,1)) on $\widetilde{X}$, since by Ehresmann's theorem, there exists a diffeomorphism between $\widetilde{X}$ and $\widetilde{\mathscr{X}_{s_i}}$.

From Lemma \ref{Kähler2} ( and the proof of Lemma \ref{kählergene}), we can consider $\widetilde{\omega}_t:=T_0\left[r^*_t(\omega_t)\right]+\left[E\right]_{|\widetilde{\mathscr{X}}_{t}}$ Kähler classes on $\widetilde{\mathscr{X}_t}$ where $(\omega_t)$ depends continuously of $t\in U$.
Then $\int_{\widetilde{\mathscr{X}_{s_i}}}\widetilde{\omega}_{s_i}^{2n-1}\cdot\widetilde{\alpha}_{i,R}$
converges to $\int_{\widetilde{\mathscr{X}_s}}\widetilde{\omega}_{s}^{2n-1}\cdot\widetilde{\alpha}_R$.
Hence by Remark \ref{remarkaddedbyreferee}, $(\widetilde{\alpha}_{i,R})$ is weakly bounded, so by Banach-Alaoglu's theorem (Remark \ref{weakcompactness})
%Hence by the sequence $(r_{s_i}^*(\alpha_i)+(\pi_2\circ j)^*(H)_{|\widetilde{\mathscr{X}_{s_i}}})$ is weakly bounded. 
we can find a subsequence $(\widetilde{\alpha}_{\mu(i),R})$ 
which converges to a closed positive $(1,1)$ current of class $\widetilde{\alpha}_R$.  

Now we consider 
\begin{equation}
\widetilde{\alpha}_R-\epsilon\widetilde{\omega}_s=r^*_s(R\alpha-\epsilon T_0\left[\omega_s\right])+(1-\epsilon)\left[E\right]_{|\widetilde{\mathscr{X}}_{s}},
\label{Kcurrent}
\end{equation}
for a small $\epsilon\in\mathbb{R}_+^*$. We can take $R\geq \max(R_0,T_0)$ and rewrite (\ref{Kcurrent}):
$$\frac{1}{1-\epsilon}\left(\widetilde{\alpha}_R-\epsilon\widetilde{\omega}_s\right)=\frac{R}{(1-\epsilon)}r^*_s(\alpha-\epsilon\left[\omega_s\right])+\left[E\right]_{|\widetilde{\mathscr{X}}_{s}},$$
with $0<\epsilon<1$ small enough such that $\alpha-\epsilon\left[\omega_s\right]\in\mathcal{C}_{\mathscr{X}_{s}}$.
Then, we can apply the same technique as before to find a real number $R_0'\geq \max(R_0,T_0)$ such that $\widetilde{\alpha}_R-\epsilon\widetilde{\omega}_s$ is 
the class of a closed positive $(1,1)$ current for all $R\geq R_0'$.
%Let $\omega_s$ be a Kähler form on $\mathscr{X}_{s}$, we can show by the same technique that for $\epsilon\in \mathbb{R}_+^*$ small enough, 
%$$\widetilde{\alpha}-\epsilon\left(r_s^*(\omega_s)+(\pi_2\circ j)^*(H)_{|\widetilde{\mathscr{X}_s}}\right)=r_s^*(\alpha-\epsilon\omega_s)+(1-\epsilon)(\pi_2\circ j)^*(H)_{|\widetilde{\mathscr{X}_s}}$$ 
%is also the class of a closed positive $(1,1)$ current. 
%However by Lemma \ref{Kähler}, $$\epsilon \left(r_s^*(\omega_s)+(\pi_2\circ j)^*(H)_{|\widetilde{\mathscr{X}_s}}\right)$$ is a Kähler class on $\widetilde{\mathscr{X}_s}$.
This proves that $\widetilde{\alpha}_R$ is the class of a Kähler current for all $R\geq R_0'$. 
\end{proof}
\begin{proof}[Proof of Theorem \ref{PC}]
If $X$ is projective, we have an ample line bundle $\mathcal{L}$ on $X$. Hence from Remark \ref{kählerpositive}, we have $B_X(c_1(\mathcal{L}),c_1(\mathcal{L}))>0$. 
Conversely, assume that we have a line bundle $\mathcal{L}$ with $c_1(\mathcal{L})\in \mathcal{C}_X$. Then by Lemma \ref{final}, we can find an integer $N\in\mathbb{N}$ and an integral combination of exceptional divisor $E$ such that
$Nr_o^*(c_1(\mathcal{L}))+\left[E\right]_{|\widetilde{X}}$ is a Kähler current on $\widetilde{X}$. 
Hence by Theorem \ref{JSth}, $\widetilde{X}$ is Moishezon. So $X$ is Moishezon. Since all the singularities of $X$ are rational (cf. Proposition \ref{MildSingu}), we can apply the Namikawa criterion \cite[Theorem 6]{Namikawa} to prove that $X$ is projective. 
\end{proof}
\section{Global Torelli theorem}\label{Gloglo}
\subsection{Hyperkähler orbifolds and Twistor space}\label{twisty}
In this section, we use Notation \ref{XUV*}.
An object $\Theta$ such as a \emph{metric}, a \emph{complex structure}, a \emph{connection} or a \emph{curvature} is defined on an orbifold $X$ as usual object on the smooth part $X^*$ satisfying the following conditions. For all local uniformizing systems $(V,G,\pi)$ of an open set $U\subset X$, there exists an object $\widetilde{\Theta}$ on $V$ invariant by the action of $G$ (or commuting with the action of G in the case of $\widetilde{\Theta}$ being a complex structure) such that $\Theta_{|U^*}$ is the image of $\widetilde{\Theta}_{|V^*}$. 
Campana has shown the following theorem.
\begin{thm}[\cite{Campana}, Theorem 4.1]\label{Ricci}
Let $X$ be a compact kähler orbifold with $c_1(X)=0$, and let $\omega$ be a Kähler class on $X$. Then $\omega$ is represented by a unique Kähler Ricci-flat metric on $X$. 
\end{thm}
Using this result, Campana in \cite[Definition 6.5]{Campana} defined a hyperkähler orbifold as follows.
\begin{defi}
A compact orbifold $X$ of dimension $2n$ is said to be \emph{hyperkähler} if $X^*$ is simply connected and if $X$ admits a Ricci flat Kähler metric $g$ such that its restriction to the smooth part $g_{|X^*}$ has holonomy $\Sp(n)$.
\end{defi}
\begin{prop}[\cite{Campana}, Proposition 6.6]\label{hypersymplec}
Let $X$ be an orbifold with $\codim \Sing X\geq 4$.
The following two statements are equivalent:
\begin{itemize}
\item[(i)]
X is a hyperkähler orbifold;
\item[(ii)]
X is an irreducible symplectic orbifold.
%\item[(iii)]
%X admits a Ricci flat metric $g$ and three complex structure $I$, $J$ and $K$ in quaternionic relation such that 
%$g$ is Kähler according to each of them. Moreover if $\alpha$ is a Kähler class on $X$, $g$ can be chosen such that $\alpha=[g(\cdot,I\cdot)]$.
\end{itemize}
\end{prop}
As a consequence, we can generalize the twistor space for irreducible symplectic orbifolds (previously constructed for smooth hyperkähler manifold in \cite[Section 3 (F)]{Hitchin} and for K3 surfaces with quotient singularities in \cite[Theorem 2]{Kobayashi}).
A \emph{positive three-space} is a subspace $W\subset \Lambda\otimes\R$ such that $B_{|W}$ is positive definite. For any positive three-space, we define the associated \emph{twistor line} $T_W$ by:
$$T_W:=\mathcal{D}\cap \mathbb{P}(W\otimes\C).$$
\begin{thm}\label{Twistor}
Let $(X,\varphi)$ be a marked irreducible symplectic orbifold with $\varphi:H^2(X,\Z)\rightarrow \Lambda$. Let $\alpha$ be a Kähler class on $X$, and $W=\Vect_{\R}(\varphi(\alpha),$ $\varphi(\Ree \sigma_X),\varphi(\Ima \sigma_X))$. 
Then:
\begin{itemize}
\item[(i)]
There exists a metric $g$ and three complex structures $I$, $J$ and $K$ in quaternionic relation on $X$ such that:
$$\alpha= \left[g(\cdot,I\cdot)\right]\ \text{and}\ g(\cdot,J\cdot)+ig(\cdot,K\cdot)\in H^{0,2}(X).$$
\item[(ii)]
There exists a deformation of $X$: 
$$\mathscr{X}\rightarrow T(\alpha)\simeq\mathbb{P}^1,$$ such that the period map
$\mathscr{P}:T(\alpha)\rightarrow T_W$ provides an isomorphism. Moreover, for each $s=(a,b,c)\in \mathbb{P}^1$, the associated fiber $\mathscr{X}_s$ is an orbifold diffeomorphic to $X$ endowed with the complex structure $aI+bJ+cK$.
\end{itemize}
\end{thm}
\begin{proof}
Let $\alpha$ be a Kähler class of $X$. By Theorem \ref{Ricci}, this class is represented by a Ricci-flat metric $g$ on $X$. 
Let $U\subset X$ with a local uniformizing system $(V,G,\pi)$. We have a metric $\widetilde{g}$ on $V$ such that its restriction to $V^*$ induces the metric $g$ on $U^*$. Moreover $\alpha=\left[g(\cdot,I\cdot)\right]$ where $I$ is the complex structure on $X$. Since $X$ is irreducible symplectic, by Proposition \ref{hypersymplec}, we can find two other complex structures $J$ and $K$ on $X^*$ in quaternionic relation with $I$, that is: $IJ=-K$. In particular, we can consider the holomorphic 2-form $\sigma:=g(\cdot,J\cdot)+ig(\cdot,K\cdot)$ on $X^*$. Let $\sigma_{|U^*}$ be its restriction to $U^*$. By (\ref{PetersLemma}) and \cite[Lemma 2.1]{Fujiki}, $\pi^*\sigma_{|U^*}$ extends to a unique non-degenerate holomorphic 2-form $\widetilde{\sigma}_{|V}$ defined on all of $V$. Hence, $\Ree \widetilde{\sigma}_{|V}$ and $\Ima \widetilde{\sigma}_{|V}$ are non-degenerate $\mathbb{R}$-bilinear forms on $TV$. So, they provide $\mathbb{R}$-linear automorphisms $\widetilde{J}$ and $\widetilde{K}$ on $TV$ such that 
$$\Ree \widetilde{\sigma}_{|V}=\widetilde{g}(\cdot,\widetilde{J}\cdot)\ \text{and}\ \Ima \widetilde{\sigma}_{|V}=\widetilde{g}(\cdot,\widetilde{K}\cdot).$$
By Remark \ref{PrillR}, $\pi:V^*\rightarrow U^*$ is a local isomorphism, hence, by construction, $\widetilde{J}$ and $\widetilde{K}$ will behave locally as $J$ and $K$. In particular, they are almost complex structures on $V^*$ in quaternionic relation with $\widetilde{I}$ the complex structure on $V$ defining $I$ on $U$.
Hence by continuity, $\widetilde{J}$ and $\widetilde{K}$ are almost complex structures, in quaternionic relation with $\widetilde{I}$, on all of $V$. With the same argument, we can see that $\widetilde{J}$ and $\widetilde{K}$ are integral.
Let $D^{\widetilde{g}}$ be the Levi-Civita connection of $\widetilde{g}$ on $V$. By construction, $\widetilde{J}$, $\widetilde{K}$, $\Ree \widetilde{\sigma}_{|V}$ and $\Ima \widetilde{\sigma}_{|V}$ are $D^{\widetilde{g}}$-parallel on $V^*$, and therefore on all of $V$ by continuity. It follows that the complex structures $J$ and $K$ on $X^*$ are actually complex structures on $X$ (in the sense explained in the beginning of the section). 

Now, let $(U_i)$ be a basis of open sets of $X$ such that all $U_i$ admit a local uniformizing system $(V_i,G_i,\pi_i)$. As we have seen, the three complex structures on $X^*$ provide three complex structures on each $V_i$. Then, we can construct the twistor space $\mathscr{V}_i\rightarrow\mathbb{P}^1$ for each $V_i$ (see \cite[Section 3 (F)]{Hitchin}). The group $G_i$ acts on $\mathscr{V}_i$ preserving the fibers; moreover, by construction, the complex structures on the $\mathscr{V}_i$ are uniquely determined by the tree complex structures $I$, $J$, $K$ on $X^*$. Hence, the quotients $\mathscr{V}_i/G_i$ will glue together into an orbifold $\mathscr{X}$ with the desired projection $\mathscr{X}\rightarrow \mathbb{P}^1$.
\end{proof}
%\begin{rmk}
%We have seen in this proof that if $X$ is an irreducible symplectic orbifold,
%X admits a Ricci flat metric $g$ and three complex structures $I$, $J$ and $K$ in quaternionic relation such that 
%$g$ is Kähler according to each of them.
%\end{rmk}
\subsection{The Kähler cone}
Let $X$ be an irreducible symplectic compact Kähler orbifold with $\codim\sing$ $ X\geq 4$. We denote by $\mathcal{K}_X$ the \emph{Kähler cone} of $X$.
We recall that the connected component of $\left\{\left.\alpha\in H^{1,1}(X,\R)\ \right|\ B_X(\alpha,\alpha)>0\right\}$ containing $\mathcal{K}_X$ is called the positive cone of $X$ and denoted by $\mathcal{C}_X$. As in the previous sections, we denote by $\mathscr{M}_{\Lambda}$ the moduli space of marked irreducible symplectic orbifolds of Beauville--Bogomolov lattice $\Lambda$.
\begin{prop}\label{Hprop5.1}
Let $(X,\varphi)\in\mathscr{M}_{\Lambda}$ be a marked irreducible symplectic orbifold. Assume that $\alpha\in \mathcal{C}_X$ is general, i.e. $\alpha$ is contained in the complement of countably many nowhere dense closed subsets. Then there exists a point $(X',\varphi')\in\mathscr{M}_{\Lambda}$, which cannot be separated from $(X,\varphi)$ such that $(\varphi'^{-1}\circ\varphi)(\alpha)\in H^2(X',\Z)$ is a Kähler class.
\end{prop}
\begin{proof}
Knowing the local Torelli theorem (Theorem \ref{LocalTorelli}), the projectivity criterion (Theorem \ref{PC}), the (1,1)-Lefschetz theorem (Proposition \ref{Lefschetz11}), bimeromorphisms between non-separated orbifolds (Proposition \ref{separated}, Remark \ref{nonsepgraph}) and the construction of the twistor space (Theorem \ref{Twistor}), the proof of Proposition \ref{Hprop5.1} can be copied word by word from the proof of \cite[Proposition 5.1]{Huybrechts5}.
\end{proof}
\begin{cor}\label{Kählercone}
 Assume that $\Pic X=0$, then 
$\mathcal{K}_{X}=\mathcal{C}_X.$
\end{cor}
\begin{proof}
Let $\varphi$ be a marking for $X$. 
From Proposition \ref{Hprop5.1}, if we consider 
$\alpha\in \mathcal{C}_X$ general, then there exists a point $(X',\varphi')\in\mathscr{M}_{\Lambda}$, which cannot be separated from $(X,\varphi)$ such that $(\varphi'^{-1}\circ\varphi)(\alpha)\in H^2(X',\Z)$ is a Kähler class.
However, when $\Pic X=0$, it follows from Proposition \ref{Lefschetz11} and \ref{separated}, that $(X',\varphi')=(X,\varphi)$. 
Hence $\alpha\in \mathcal{K}_X$. So $\mathcal{K}_X$ is a dense open convex subset of the convex open set $\mathcal{C}_X$. It follows that $\mathcal{C}_X=\mathcal{K}_X$.
\end{proof}
\subsection{Conclusion}
Using the work from the previous sections we can finally prove the main result of this article. 
As before, $\mathscr{M}_{\Lambda}$ is the moduli space of marked irreducible symplectic orbifolds of Beauville--Bogomolov lattice $\Lambda$ and $\overline{\mathscr{M}_{\Lambda}}$ its Hausdorff reduction constructed in Section \ref{modu}.
\begin{rmk}\label{Period domain}
Let $X$ be an irreducible symplectic orbifold. 
From Theorem \ref{LocalTorelli}, the Beauville--Bogomolov lattice $H^2(X,\Z)$ has signature $(3,b_2(X)-3)$. It follows that the properties of the period domain $\mathcal{D}$ in the smooth case stated in \cite[Section 3]{Huybrechts6} remain true in our case. 
\end{rmk}

%From , we deduce with exactly the same technique as Huybrechts , the surjectivity of the period map. 
\begin{prop}\label{surject}
Assume that $\mathscr{M}_{\Lambda}\neq\emptyset$ and let $\mathscr{M}_{\Lambda}^{°}$ be a connected component of the moduli space $\mathscr{M}_{\Lambda}$.
Then the period map:
$$\mathscr{P}:\mathscr{M}_{\Lambda}^{°}\rightarrow \mathcal{D}$$ is surjective.
\end{prop}
\begin{proof}
By Remark \ref{Period domain}, \cite[Proposition 3.7]{Huybrechts6} is also true in our case.
Hence knowing Theorem \ref{Twistor} and Corollary \ref{Kählercone}, the proof of the surjectivity of the period map can be copied word by word from the proof of \cite[Theorem 5.5]{Huybrechts6}.
\end{proof}
\begin{thm}\label{torelliglobal}
Assume that $\mathscr{M}_{\Lambda}\neq\emptyset$ and let $\mathscr{M}_{\Lambda}^{°}$ be a connected component of the moduli space $\mathscr{M}_{\Lambda}$.
Then $\mathscr{P}: \overline{\mathscr{M}_{\Lambda}}^{°}\rightarrow \mathcal{D}$ is an isomorphism.
\end{thm}
\begin{proof}
By Remark \ref{Period domain}, \cite[Proposition 3.10 and Lemma 3.11]{Huybrechts6} remains true in our case. Hence,
with exactly the same proof as found in \cite[Section 5.4]{Huybrechts6}, we deduced from general consideration about covering spaces (\cite[Proposition 5.6]{Huybrechts6}), Corollaries \ref{Hausdorff}, \ref{Kählercone} and Theorem \ref{Twistor} 
that $\mathscr{P}: \overline{\mathscr{M}_{\Lambda}}^{°}\rightarrow \mathcal{D}$ is a covering space. Since $\mathcal{D}$ is simply connected (\cite[Proposition 3.1]{Huybrechts6}), $\mathscr{P}$ is actually an isomorphism.
\end{proof}
\begin{rmk}
From Proposition \ref{separated}, we also obtain that $\mathscr{P}: \mathscr{M}_{\Lambda}^{°}\rightarrow \mathcal{D}$ is generically injective. 
\end{rmk}
\begin{cor}
Let $X$ and $X'$ be two irreducible symplectic orbifolds such that there exists a parallel transform operator $\lambda:H^2(X,\Z)\simeq H^2(X',\Z)$ which is a Hodge isometry. Then $X$ and $X'$ are bimeromorphic. 
\end{cor}
\begin{proof}
Let $\varphi:H^2(X,\Z)\rightarrow \Lambda$ be a marking of $X$ and $\varphi':=\varphi\circ\lambda^{-1}$. Since $\lambda$ is a parallel transport operator, $(X,\varphi)$ and $(X',\varphi')$ are in the same connected component of $\mathscr{M}_{\Lambda}$. Since $\lambda$ is a Hodge isometry, we have $\mathscr{P}(X,\varphi)=\mathscr{P}(X',\varphi')$. Hence by Corollary \ref{torelliglobal}, they are non-separated points in $\mathscr{M}_{\Lambda}$. So by Proposition \ref{separated}, they are bimeromorphic. 
\end{proof}
\begin{rmk}
For the converse assertion, one would need more knowledge about bimeromorphisms between orbifolds. This will be the object of future works. 
\end{rmk}
\bibliographystyle{amssort}

\noindent
Gr\'egoire \textsc{Menet}

\noindent
Institut Fourier

\noindent 
100 rue des Math\'ematiques, Gi\`eres (France),

\noindent
{\tt gregoire.menet@univ-grenoble-alpes.fr}

\end{document}